\documentclass[11pt,letterpaper,oneside]{amsart}

\usepackage{etex}
  \usepackage{amsmath,amssymb,amsthm}
  \usepackage{amscd}
  \usepackage{geometry}
  \usepackage{graphicx,epstopdf}
  \usepackage{color}
  \usepackage{enumerate}
  \usepackage{xspace}
  \usepackage{tipa}
\usepackage{tikz}
\usepackage{mathrsfs}
\usetikzlibrary{matrix}
\usepackage{tikz-cd}
\usetikzlibrary{positioning}
\usetikzlibrary{cd}
\usepackage[all]{xy}
\DeclareMathAlphabet{\pazocal}{OMS}{zplm}{m}{n}
\tikzset{>=stealth}
  \usepackage{epsfig}
  \usepackage{pinlabel}
  \usepackage[all]{xy}

  \usepackage{hyperref}

\newcommand{\cl}[1]{\marginpar{\color{purple}\tiny #1 --cl}}

\newcommand{\stab}{{\rm{Stab}}}
  \newcommand{\calA}{\mathcal{A}}
  \newcommand{\calB}{\mathcal{B}}
  \newcommand{\calC}{\mathcal{C}}
  
  \newcommand{\calE}{\mathcal{E}}
  \newcommand{\calF}{\mathcal{F}}
  \newcommand{\calG}{\mathcal{G}}
  \newcommand{\calH}{\mathcal{H}}

  \newcommand{\calL}{\mathcal{L}}
  
  \newcommand{\calN}{\mathcal{N}}
  
  \newcommand{\calP}{\mathcal{P}}

  \newcommand{\calS}{\mathcal{S}}

  \newcommand{\calW}{\mathcal{W}}
  \newcommand{\calX}{\mathcal{X}}

\newcommand{\ev}{{\rm{ev}}}

  \newcommand{\HH}{\mathbb{H}}

  \newtheorem{theorem}{Theorem}[section]
  \newtheorem{proof of the main theorem}[theorem]{Proof of the Main Theorem}
  \newtheorem{proposition}[theorem]{Proposition}

  \newtheorem{corollary}[theorem]{Corollary}
  \newtheorem{lemma}[theorem]{Lemma}
  
  \newtheorem*{conjecture*}{Conjecture}

\newtheorem*{CTUtheorem}{Theorem~\ref{CTU}}

  \theoremstyle{definition}
  \newtheorem{definition}[theorem]{Definition}
   \newtheorem{remark}[theorem]{Remark}
  \newtheorem*{remark*}{Remark}
  
  \newtheorem{claim}[theorem]{Claim}
  \newtheorem*{claim*}{Claim}

  \newtheorem*{question*}{Question}
  \newtheorem*{answer*}{Answer}
  \newtheorem*{application*}{Application}
  \newtheorem*{ideas*}{ideas}

  \DeclareMathOperator{\PMod}{PMod}

  \DeclareMathOperator{\diam}{diam}

  \DeclareMathOperator{\Diff}{Diff}

  \newcommand{\ML}{\ensuremath{\mathcal{ML}}\xspace}
  \newcommand{\FL}{\ensuremath{\mathcal{FL}}\xspace}
  \newcommand{\EL}{\ensuremath{\mathcal{EL}}\xspace}
  \newcommand{\PML}{\ensuremath{\mathcal{PML}}\xspace}
         \newcommand{\PFL}{\ensuremath{\mathcal{PFL}}\xspace}
\newcommand{\sEL}{\ensuremath{\calE\calL^s(\dot S)}\xspace}

  \newcommand{\la}{\langle} 
  \newcommand{\ra}{\rangle}
  \newcommand{\co}{\colon\thinspace}

  \newcommand{\param}{{\mathchoice{\mkern1mu\mbox{\raise2.2pt\hbox{$
  \centerdot$}}
  \mkern1mu}{\mkern1mu\mbox{\raise2.2pt\hbox{$\centerdot$}}\mkern1mu}{
  \mkern1.5mu\centerdot\mkern1.5mu}{\mkern1.5mu\centerdot\mkern1.5mu}}}
\renewcommand{\setminus}{{\smallsetminus}}

    \newcommand{\sC}{{\mathcal C^s(\dot S)}}
    \newcommand{\csC}{{\overline{\mathcal C^s}(\dot S)}}
    \newcommand{\C}{{\mathcal C}}
    \newcommand{\Cdot}{\mathcal C(\dot S)}
  
\newcommand{\lcut}{{ \{ \hspace{-.1cm} \{ }}
\newcommand{\rcut}{{ \} \hspace{-.1cm} \} }}

\newcommand{\BS}{\mathbb S}

\newcommand{\Abd}{\BS^1_{\hspace{-.1cm}\calA}}
\newcommand{\AbdC}{\BS^1_{\hspace{-.1cm}\calA_0}}

\newcommand{\id}{{\rm{id}}}

\begin{document}


\title[A Cannon-Thurston map for surviving complexes]{A Universal Cannon-Thurston map and the surviving curve complex.}

  \author {Funda G\"ultepe} \thanks{The first author was partially supported by a University of Toledo startup grant.}
  \author {Christopher J. Leininger} \thanks{The second author was partially supported by NSF grant DMS-1510034, 1811518, and 2106419.}
\author{Witsarut Pho-on}

\address{Department of Mathematics and Statistics \\
 University of Toledo\\
 Toledo, OH, 43606}
\email{funda.gultepe@utoledo.edu}
\urladdr{http://www.math.utoledo.edu/~fgultepe/}

\address{Department of Mathematics\\
Rice University\\ Houston, TX 77005}
\email{cjl12@rice.edu}
\urladdr{https://sites.google.com/view/chris-leiningers-webpage/home}

\address{Department of Mathematics\\Faculty of Science\\Srinakharinwirot University\\ Bangkok 10110, Thailand}
\email{witsarut@g.swu.ac.th}
\urladdr{https://sites.google.com/g.swu.ac.th/witsarut/}

\begin{abstract} Using the Birman exact sequence for pure mapping class groups,  we construct a universal Cannon-Thurston map onto the boundary of a curve complex for a surface with punctures we call {\em surviving curve complex}.  Along the way we prove hyperbolicity of this complex and identify its boundary as a space of laminations.  As a corollary we obtain a universal Cannon-Thurston map to the boundary of the ordinary curve complex, extending earlier work of the second author with Mj and Schleimer.
\end{abstract}

\maketitle

\section{introduction} \label{S:intro}

Given a closed hyperbolic $3$--manifold $M$ that fibers over the circle with fiber a surface $S$, Cannon and Thurston \cite{CT} proved that the lift to the universal covers $\HH^2 \to \HH^3$ of the inclusion $S \to M$ extends to a continuous $\pi_1(S)$-equivariant map of the compactifications.  This is quite remarkable as the ideal boundary map $\BS^1_\infty \to \BS^2_\infty$ is a $\pi_1S$--equivariant, sphere--filling Peano curve.   A {\em Cannon-Thurston map}, $\BS^1_\infty \to \BS^2_\infty$, for a type-preserving, properly discontinuous actions of the fundamental group $\pi_1S$ of hyperbolic surfaces (closed or punctured) acting on hyperbolic $3$--space was shown to exist in various situations (see \cite{Minsky-rigidity,ADP-CT,McMullen,BowCT1}), with Mj \cite{Mj1} proving the existence in general (see Section~\ref{S:historical} for a discussion of even more general Cannon-Thurston maps).



Suppose that $S$ is a hyperbolic surface with basepoint $z \in S$, and write $\dot S = S \smallsetminus\{z\}$. The curve complex of $\dot S$ is a $\delta$--hyperbolic space on which $\pi_1S = \pi_1(S,z)$ acts via the Birman exact sequence.  In \cite{LeinMjSch}, the second author, Mj, and Schleimer constructed a {\em universal Cannon-Thurston map} when $S$ is a closed surface of genus at least $2$.  Here we complete this picture, extending this to all surfaces $S$ with complexity $\xi(S) \geq 2$.

\begin{theorem}[Universal Cannon-Thurston Map] \label{T:UCT C short} Let $S$ be a connected, orientable surface with $\xi(S) \geq 2$.  Then there exists a subset $\AbdC \subset \BS^1_\infty$ and a continuous, $\pi_1S$--equivariant, finite-to-one surjective map $\partial \Phi_0 \colon \AbdC \to \partial \C(\dot S)$.  Moreover, if $\partial i \colon \BS^1_\infty \to \BS^2_\infty$ is any Cannon-Thurston map for a proper, type-preserving, isometric action on $\mathbb H^3$ without accidental parabolics, then there exists a map $q \colon \partial i(\AbdC) \to \partial \C(\dot S)$ so that $\partial \Phi_0$ factors as
\[
\xymatrix{ \AbdC \ar[r]_{\partial i \quad} \ar@/^1pc/[rr]^{\partial \Phi_0} & \partial i(\AbdC) \ar[r]_{q} & \partial \C(\dot S).
}\]
\end{theorem}

For the reader familiar with Cannon-Thurston maps in the setting of cusped hyperbolic surfaces, the finite-to-one condition may seem unnatural.  We address this below in the process of describing the subset $\AbdC \subset \BS^1_\infty$.  First, we elaborate on the universal property of the theorem (that is, the ``moreover" part).

Let $p \colon \HH = \HH^2 \to S$ denote the universal cover \footnote{We will mostly be interested in real hyperbolic space in dimension $2$, so will simply write $\HH = \HH^2$.}.  A proper, type-preserving, isometric action of $\pi_1S$ on a $\mathbb H^3$ has quotient hyperbolic $3$--manifold homeomorphic to $S \times \mathbb R$.  Each of the two ends (after removing cusp neighborhoods) is either geometrically finite or simply degenerate.  In the latter case, there is an associated {\em ending lamination} that records the asymptotic geometry of the end; see \cite{TNotes,bonahon,Minsky-endinglam,BCM-endinglam}].  The Cannon-Thurston map $\BS^1_\infty \to \BS^2_\infty$ for such an action is an embedding if both ends are geometrically finite; see \cite{Floyd}.  If there are one or two degenerate ends, the Cannon-Thurston map is a quotient map onto a dendrite or the entire sphere $\BS^2_\infty$, respectively, where a pair of points $x,y \in \BS^1_\infty$ are identified if and only if $x$ and $y$ are ideal endpoints of a leaf or complementary region of the $p^{-1}(\calL)$ for (one of) the ending lamination(s) $\calL$; see \cite{CT,Minsky-rigidity,BowCT1,Mj2}.  A more precise version of the universal property is thus given by the following.  Here $\EL(S)$ is the space of ending laminations of $S$, which are all possible ending laminations of ends of hyperbolic $3$--manifolds as above; see Section~\ref{S:laminations} for definitions.

\begin{theorem} \label{T:Phi0 identified} Given two distinct points $x,y\in \AbdC$, $\partial\Phi_0(x) = \partial \Phi_0(y)$ if and only if $x$ and $y$ are the ideal endpoints of a leaf or complementary region of $p^{-1}(\calL)$ for some $\calL \in \EL(S)$.
\end{theorem}

When $S$ has punctures, $\partial \C(\dot S)$ is not the most natural ``receptacle" for a universal Cannon-Thurston map.  Indeed, there is another hyperbolic space whose boundary naturally properly contains $\partial \C(\dot S)$.  The {\em surviving curve complex} of $\dot S$, denoted $\C^s(\dot S)$ is the subcomplex of $\C(\dot S)$ spanned by curves that ``survive" upon filling $z$ back in.  In section \ref{S:hyperbolicity}, we prove that $\C^s(\dot S)$ is hyperbolic.  One could alternatively verify the axioms due to Masur and Schleimer \cite{MSch1}, or try to relax the conditions of Vokes \cite{Vokes} to prove hyperbolicity; see Section~\ref{S:hyperbolicity}.

The projection $\Pi \colon \C^s(\dot S) \to \C(S)$ was studied by the second author with Kent and Schleimer in \cite{LeinKentSch} where it was shown that for any vertex $v \in \C(S)$, the fiber $\Pi^{-1}(v)$ is $\pi_1S$--equivariantly isomorphic to the Bass-Serre tree dual to the splitting of $\pi_1S$ defined by the curve determined by $v$; see also \cite{Harer},\cite{Hat-Vogt}.  As such, there is a $\pi_1S$--equivariant map $\Phi_v \colon \HH \to \Pi^{-1}(v) \subset \C^s(\dot S)$; see \S\ref{S:tree map construction}.  As we will see, the first part of Theorem~\ref{T:UCT C short} is a consequence of the following; see Section~\ref{S:UCT maps}.

\newcommand{\CTstatement}{
For any vertex $v\in \C$, the map $\Phi_v: \HH \rightarrow \C^s(\dot S)$ has a continuous $\pi_1(S)$--equivariant extension
 \[\overline\Phi_v: \HH \cup \Abd \rightarrow \overline\C^s(\dot S)\]
 and the induced  map
 \[ \partial\Phi = \overline\Phi_v|_{\Abd}: \Abd \rightarrow \partial \C^s(\dot S)\]
is surjective and does not depend on $v$. Moreover,
$\partial \Phi$ is equivariant with respect to the action of the pure mapping class group $\PMod(\dot S)$.}

\begin{theorem} \label{CT} \CTstatement
\end{theorem}

The subset $\Abd \subset \BS^1_\infty$ is defined analogously to the set $\mathbb A \subset \BS^1_\infty$ in \cite{LeinMjSch}.  Specifically, $x \in \Abd$ if and only if any geodesic ray $r \subset \HH$ starting at any point and limiting to $x$ at infinity has the property that every essential simple closed curve $\alpha \subset S$ has nonempty intersection with $p(r)$; see Section~\ref{S:UCT maps}. It is straightforward to see that $\Abd$ is the largest set on which a Cannon-Thurston map can be defined to $\partial \C^s(\dot S)$.

As we explain below, $\AbdC \subsetneq \Abd$ and a pair of points in $\AbdC$ are identified by $\partial \Phi_0$ if and only if they are identified by $\partial \Phi$, and thus $\partial \Phi$ is also finite-to-one on $\AbdC$.  It turns out that this precisely describes the difference between $\Abd$ and $\AbdC$.   Let $Z \subset \partial \C^s(\dot S)$ be the set of points $x$ for which $\partial \Phi^{-1}(x)$ is infinite.
\begin{proposition} \label{P:CT difference} We have $\displaystyle \Abd \setminus \AbdC = \partial \Phi^{-1}(Z)$. 
\end{proposition}

The analogue of Theorem~\ref{T:Phi0 identified} is also valid for $\Phi$.

\begin{theorem}\label{CTU}
Given two distinct points $x,y\in \Abd$, $\partial\Phi(x) = \partial \Phi(y)$ if and only if $x$ and $y$ are the ideal endpoints of a leaf or complementary region of $p^{-1}(\calL)$ for some $\calL \in \EL(S)$.
\end{theorem}

It is easy to see that for any ending lamination $\calL \in \EL(S)$, the endpoints at infinity of any leaf of $p^{-1}(\calL)$ (and hence also the non parabolic fixed points of complementary regions) are contained in $\Abd$, though this a fairly small subset; for example, almost-every point $x \in \BS^1_\infty$ has the property that any geodesic ray limiting to $x$ has dense projection to $S$.   The complementary regions that contain parabolic fixed points are precisely the regions with infinitely many ideal endpoints.
Together with Proposition~\ref{P:CT difference} provides another description of the difference $\Abd \setminus \AbdC$; see Corollary~\ref{C:CT difference 2}.

A important ingredient in the proofs of the above theorems is an identification of the Gromov boundary $\partial \C^s(\dot S)$, analogous to Klarreich's Theorem \cite{Klarreich}; see Theorem~\ref{T:Klarreich}.  Specifically, we let $\calE\calL^s(\dot S)$ denote the space of ending laminations on $\dot S$ together with ending laminations on all proper {\em witnesses} of $\dot S$; see Section~\ref{S:witnesses defined}.  We call $\EL^s(\dot S)$ the space of {\em surviving ending laminations}.  A more precise statement of the following is proved in Section~\ref{S:boundary}; see Theorem~\ref{T:boundary ending precise}

\begin{theorem}\label{survivalendinglam} There is a $\PMod(\dot S)$--equivariant homeomorphism $\partial \C^s(\dot S) \to \EL^s(\dot S)$.
\end{theorem}

To describe the map $\partial \Phi_0$ in Theorem~\ref{T:UCT C short} we consider the map $\partial \Phi \colon \Abd \to \partial \C^s(\dot S)$ from Theorem~\ref{CT}, composed with the homeomorphism $\partial \C^s(\dot S) \to \EL^s(\dot S)$ from Theorem~\ref{survivalendinglam}.  Since $\EL(\dot S)$ is a subset of $\EL^s(\dot S)$, we can simply take $\AbdC \subset \Abd$ to be the subset that maps onto $\EL(\dot S)$, and compose the restriction $\partial \Phi$ to this subset with the homeomorphism $\EL(\dot S) \to \partial \C(\dot S)$ from Klarreich's Theorem.  The more geometric description of $\AbdC$ is obtained by a more detailed analysis of the map $\partial \Phi$ carried out in Section~\ref{S:UCT maps}.

\subsection{Historical discussion} \label{S:historical}
Existence of the Cannon-Thurston map in the context of Kleinian groups is proved by several authors starting with Floyd \cite{Floyd} for geometrically finite Kleinian groups and then by Cannon and Thurston for fibers of closed hyperbolic 3-manifolds fibering over the circle.  Cannon and Thurston's work was circulated as a preprint around 1984 and inspired works of many others before it was published in 2007 \cite{CT}. The existence of the Cannon-Thurston map was proven by Minsky \cite{Minsky-rigidity} for closed surface groups of bounded geometry and by by Mitra and Klarreich \cite{Mitra2, Klarreich2} for hyperbolic 3-manifolds of bounded geometry with an
incompressible core and without parabolics. Alperin-Dicks-Porti \cite{ADP-CT} proved the existence of the Cannon-Thurston map for figure eight knot complement, McMullen \cite{McMullen} for punctured torus groups, and then Bowditch \cite{BowCT1,BowCT2} for more general punctured surface groups of bounded geometry. Mj completed the investigation for all finitely generated Kleinian surface groups without accidental parabolics, first for closed and then for punctured surfaces in a series of papers that culminated in the two papers \cite{Mj1} and \cite{Mj2}, the latter with an appendix by S.~Das.  For general Kleinian groups, see Das-Mj \cite{DasMj} and Mj \cite{MjKleinian}, and the survey \cite{Mj-survey}.

Moving beyond real hyperbolic spaces, it is now classical that a quasi-isometric embedding of one Gromov hyperbolic space into another extends to an embedding of the Gromov boundaries.
One of the first important generalizations of Cannon and Thurston's work outside the setting of Kleinian groups is due to Mitra in \cite{Mitra1} who proved that given a short exact sequence
\[1\rightarrow H \rightarrow \Gamma \rightarrow G \rightarrow 1\]
of infinite word hyperbolic groups, the Cannon--Thurston map exists and it is surjective. In this case the Cannon-Thurston map $\partial H \rightarrow \partial \Gamma$ is defined between the Gromov boundary $\partial H$ of the fiber group $H$ and the Gromov boundary $\partial \Gamma$ of its extension $\Gamma$.  Mitra defined an algebraic ending lamination associated to points in the Gromov boundary of the base group $G$ in \cite{Mitra-ending}, and recent work of Field \cite{field} proves that the quotient of $\partial H$ in terms of such an ending lamination is a dendrite (compare the Kleinian discussion above).

In a different direction, Mitra later extended his existence result to trees of hyperbolic spaces; see \cite{Mitra2}.  In 2013 Baker and Riley gave the first example example of a hyperbolic subgroup of a hyperbolic group with no continuous Cannon-Thurston map (\cite{BakerRiley1}); see also Matsuda \cite{Matsuda}.  On the other hand, Baker and Riley (\cite{BakerRiley2}) proved existence of Cannon-Thurston maps even under \emph{arbitrarily heavy distortion} of a free subgroup of a hyperbolic group.

For free groups and their hyperbolic extensions, Cannon-Thurston maps are better understood than arbitrary hyperbolic extensions.  Kapovich and Lustig characterized the Cannon-Thurston maps for hyperbolic free-by-cyclic groups with fully irreducible monodromy \cite{CTLustKapo}. Later Dowdall, Kapovich and Taylor characterized Cannon-Thurston  maps for hyperbolic extensions of free groups coming from  convex cocompact subgroups of outer automorphism group of the free group \cite{CTDowKapo}.

Finally we note that we have only discussed a few of the many results on the existence or structure of Cannon-Thurston maps in various settings.  For more see e.g.~\cite{Mj2,MjRa,MjRel,CTLeinKapo,GUER,Fenley-SC,Frankel,Fenley-QG,Mousley}).

\subsection{Outline} \label{S:outline} 
In Section~\ref{S:preliminaries}, we give preliminaries on curve complexes, witnesses and Gromov boundary of a hyperbolic space along with basics of spaces of laminations. In  particular, subsection~\ref{S:tree map construction} is devoted to the construction of the survival map and in subsection~\ref{S:cusps and witnesses} the relation between cusps and witnesses via the survival map is given. In Section~\ref{S:survival paths}, we define survival paths in $\sC$ and give an upper bound on the survival distance $d^s$ in terms of projection distances into curve complexes of witnesses. In Section~\ref{S:hyperbolicity} we prove the hyperbolicity of $\sC$. Section~\ref{S:distance formula} is devoted to the distance formula for $\sC$, a-la Masur-Minsky, and as a result we prove that survival paths are uniform quasi-geodesics in $\sC$.
In Section~\ref{S:boundary} we explore the boundary of the survival curve complex $\sC$ and prove that it is homeomorphic to the space of survival ending laminations on $\dot S$, a result analogous to that of Klarreich \cite{Klarreich}. In Section~\ref{S:extended survival} we extend the definition of survival map to the closures of curve complexes. Finally in Section~\ref{S:UCT maps}, we prove Theorem~\ref{CT} and the rest of the theorems from the introduction.  Specifically, we prove the existence and continuity of the map $\partial \Phi$ in Section~\ref{S:existence} and its surjectivity in Section~\ref{S:surjectivity}.  Finally, we Section~\ref{S:Universal} we prove the universal property of $\partial \Phi$ as well as constructing the map $\partial \Phi_0$.

\subsection*{Acknowledgements} The authors would like to thank Saul Schleimer for helpful conversations in the early stages of this work.  The second author would also like to thank Autumn Kent, Mahan Mj, and Saul Schleimer for their earlier collaborations that served as partial impetus for this work.

\section{Preliminaries} \label{S:preliminaries}

Throughout what follows, we assume $S$ is surface of genus $g \geq 0$ with $n \neq 0$ punctures, and complexity $\xi(S) = 3g-3+n \geq 2$.   We fix a complete hyperbolic metric of finite area on $S$ and a locally isometric universal covering $p \colon \HH \to S$.  We also fix a point $z \in S$, and write $\dot S$ to denote either the punctured surface $S \setminus \{z\}$ or the surface with an additional marked point $(S,z)$, with the situation dictating the intended meaning when it makes a difference.  We sometimes refer to the puncture produced by removing $z$ as the {\em $z$--puncture}.  We further choose $\tilde z \in p^{-1}(z) \subset \HH$ and use this to identify $\pi_1S = \pi_1(S,z)$ with the covering group of $p \colon \HH \to S$, acting by isometries.

\vspace{-.25cm}

\subsection{Notation and conventions} Let $x,y,C,K \geq 0$ with $K \geq 1$.  We write $x \stackrel{K,C}\preceq y$ to mean $x \leq Ky + C$.  We also write
\vspace{-.2cm}
\[  x \stackrel{K,C}{\asymp} y \qquad \Longleftrightarrow \qquad x \stackrel{K,C}\preceq y \, \, \mbox{ and } x \, \, \stackrel{K,C}\succeq y.\]
When the constants are clear from the context or independent of any varying quantities and unimportant, we also write $x \preceq y$ as well as $x \asymp y$.  In addition, we will use the shorthand notation $\lcut x\rcut_C$ denote the cut-off function giving value $x$ if $x \geq C$ and $0$ otherwise.

Any connected simplicial complex will be endowed with a path metric obtained by declaring each simplex to be a regular Euclidean simplex with side lengths equal to $1$.  The vertices of a connected simplicial complex will be denoted with a subscript $0$, and the distance between vertices will be an integer computed as the minimal length of a path in the $1$--skeleton.  By a {\em geodesic} between a pair of vertices $v,w$ in a simplicial complex, we mean either an isometric embedding of an interval into the $1$--skeleton with endpoints $v$ and $w$ or the vertices encountered along such an isometric embedding, with the situation dictating the intended meaning.

\subsection{Curve complexes} \label{S:curve complexes defined}

By a {\em curve} on a surface $Y$, we mean an essential (homotopically nontrivial and nonperipheral), simple closed curve.  We often confuse a curve with its isotopy class.  When convenient, we take the geodesic representative with respect to a complete finite area hyperbolic metric on the surface with geodesic boundary components (if any) and realize an isotopy class by its unique geodesic representative.  A {\em multi-curve} is a disjoint union of pairwise non-isotopic curves, which we also confuse with its isotopy class and geodesic representative when convenient.

The curve complex of a surface $Y$ with $2 \leq \xi(Y) < \infty$ is the complex $\C(Y)$ whose vertices are curves (up to isotopy) and whose $k$--simplices are multi-curves with $k+1$ components.
According to work of Masur-Minsky \cite{MM1}, curve complexes are Gromov hyperbolic.  For other proofs, see \cite{BOWhyp,Ham-hyp} as well as \cite{Aougab,Bowhyp2,ClayRafiSch,HPW} which prove uniform bounds on $\delta$.
\begin{theorem} \label{T:C hyperbolic}
For any surface $Y$, $\C(Y)$ is $\delta$--hyperbolic, for some $\delta > 0$.
\end{theorem}

The {\em surviving complex }$\C^s(\dot S)$ is defined to be the subcomplex of the curve complex $\C(\dot S)$, spanned by those curves that do not bound a twice-punctured disk, where one of the punctures is the $z$--puncture.  Given curves $\alpha,\beta \in \C^s_0(\dot S)$, we write $d^s(\alpha,\beta)$ for the distance between $\alpha$ and $\beta$ (in the $1$--skeleton).

\subsection{Witnesses for $\sC$ and subsurface projection to witnesses} \label{S:witnesses defined}

A {\em subsurface} of $\dot S$ is either $\dot S$ itself or a component $Y \subset \dot S$ of the complement of a small, open, regular neighborhood of a (representative of a) multi-curve $A$; we assume $Y$ is not a pair of pants (a sphere with three boundary components/punctures).  The boundary of $Y$, denoted $\partial Y$, is the sub-multi-curve of $A$ consisting of those components that are isotopic into $Y$.  As with (multi-)curves,  subsurfaces is considered up to isotopy, in general, but when convenient we will choose a representative of the isotopy class without comment.

\begin{definition} A {\em witness} for $\sC$ is a subsurface $W \subset \dot S$ such that for every curve $\alpha$ in $\dot S$, no representative of the isotopy class of $\alpha$ can be made disjoint from $W$.
\end{definition}
\begin{remark} Witnesses were introduced in a more general setting by Masur and Schleimer in \cite{MSch1} where they were called {\em holes}.
\end{remark}

Clearly, $\dot S$ is a witness.  Note that if $\beta$ is the boundary of a twice-punctured disk $D$, one of which is the $z$--puncture, and the complementary component $W \subset S$ with $\partial W = \beta$ is a witness.  To see this, we observe that any curve $\alpha$ in  $\sC$ that can be isotoped disjoint from $W$ must be contained in $D$, but the only such curve in $\dot S$ is $\beta$.  It is clear that these two types of subsurfaces account for all witnesses.  We let $\Omega(\dot S)$ denote the set of witnesses and $\Omega_0(\dot S) = \Omega(\dot S) \smallsetminus \{\dot S\}$ the set of proper witnesses.  We note that any proper witness $W$ is determined by its boundary curve, $\partial W$: if $W \neq \dot S$, then $W$ is the closure of the component of $\dot S \setminus \partial W$ not containing the $z$--puncture.

An important tool in what follows is the {\em subsurface projection} of curves in $\sC$ to witnesses; see \cite{MM2}.
\begin{definition} (Projection to witnesses) Let $W\subseteq \dot S$ be a witness for $\sC$ and $\alpha \in \C^s_0(\dot S)$ a curve.  We define the {\em projection of $\alpha$ to $W$}, $\pi_W(\alpha)$ as follows.  If $W = \dot S$, then $\pi_W(\alpha) = \alpha$.  If $W \neq \dot S$, then $\pi_W(\alpha)$ is the {\em set} of curves
\[ \pi_W(\alpha)= \bigcup \partial(\calN(\alpha_0 \cup \partial W)).\]
where (1) we have taken representatives of $\alpha$ and $W$ so that $\alpha$ and $\partial W$ intersect  transversely and minimally, (2) the union is over all complementary arcs $\alpha_0$ of $\alpha \setminus \partial W$ that meet $W$, (3) $\calN(\alpha_0 \cup \partial W)$ is a small regular neighborhood of of the union, and (4) we have discarded any components of $\partial (\calN (\alpha_0 \cup \partial W))$ that are not essential curves in $W$.
The projection $\pi_W(\alpha)$ is always a subset of $\C(W)$ with diameter at most $2$; see \cite{MM2}. We note that $\pi_W(\alpha)$ is never empty by definition of a witness.
\end{definition}

Given $\alpha,\beta \in \C^s_0(\dot S)$ and a witness $W$, we define the {\em distance between $\alpha$ and $\beta$ in $W$} by
\[ d_W(\alpha,\beta) = \diam\{\pi_W(\alpha) \cup \pi_W(\beta)\}.\]
Note that if $W = \dot S$, then $d_{\dot S}(\alpha,\beta)$ is simply the usual distance between $\alpha$ and $\beta$ in $\C(\dot S)$.
According to \cite[Lemma~2.3]{MM2}, projections satisfy a $2$--Lipschitz projection bound.
\begin{proposition} \label{P:2 Lipschitz} For any two distinct curves $\alpha,\beta \in \sC$, $d_W(\alpha,\beta) \leq 2 d^s(\alpha,\beta)$.  In fact, for any path $v_0,\ldots,v_n$ in $\C(\dot S)$ connecting $\alpha$ to $\beta$, such that $\pi_W(v_j) \neq \emptyset$ for all $j$, $d_W(\alpha,\beta) \leq 2n$.
\end{proposition}
We should mention that in \cite{MM2} Masur and Minsky consider the map from $\C(\dot S)$ and proves the second statement. Since $\C^s(\dot S)$ is a subcomplex of $\C(\dot S)$ for which every curve has non-empty projection, the first statement follows from the second.

We will also need the following important fact about projections from \cite{MM2}.
\begin{theorem} [Bounded Geodesic Image Theorem]  \label{BGIT} Let $\calG$ be a geodesic in $\C(Y)$ for some surface $Y$, all of whose vertices intersect a subsurface $W\subset Y$. Then, there exists a number $M>0$ such that,
\[\diam _{W}(\pi_W(\calG)) < M\]
where $\pi_W(\calG) $ denotes the image of the geodesic $\calG$ in $W$.
\end{theorem}
We assume (as we may) that $M \geq 8$, as this makes some of our estimates cleaner.  In fact, there is a uniform $M$ that is independent of $Y$ in this theorem, given by Webb~\cite{Webb-BGIT}.

\subsection{Construction of the survival map} \label{S:tree map construction} Consider the forgetful map
\[\Pi:{\sC}\rightarrow \C(S)\]
induced from the inclusion $\dot S \rightarrow S$.  By definition of $\sC$, $\Pi$ is well defined since every curve in $\sC$ determines a curve in $\C(S)$.  Each point  $v \in \C(S)$ determines a weighted multi-curve: $v$ is contained in the interior of a unique simplex, which corresponds to a multi-curve on $S$, and the barycentric coordinates determine weights on the components of the multi-curve.  According to \cite{LeinKentSch}, the fiber of the map $\Pi$ is naturally identified with the Bass-Serre tree associated to the corresponding weighted multi-curve: $\Pi^{-1}(v)= T_v$.

An important tool in our analysis is the \textit{survival map}
\[\Phi: \C(S)\times \mathbb{H}\rightarrow \sC.\]
The construction of the analogous map when $S$ is closed is described in \cite{LeinMjSch}.  Since there are no real subtleties that arise, we describe enough of the details of the construction for our purposes, and refer the reader to that paper for details.
Before getting to the precise definition of $\Phi$, we note that for every $v \in \C(S)$, the restriction of $\Phi$ to $\HH \cong \{v\} \times \HH$ will be denoted $\Phi_v \colon \HH \to \C^s(\dot S)$, and this is simpler to describe: $\Phi_v$ is a $\pi_1S$--equivariantly factors as $\HH \to T_v \to \Pi^{-1}(v)$, where the action of $\pi_1S$ on $\HH$ comes from our reference hyperbolic structure on $S$, the associated covering map $p \colon \HH \to S$, and choice of basepoint $\widetilde z \in p^{-1}(z)$.

To describe $\Phi$ in general, it is convenient to construct a more natural map from which $\Phi$ is defined as the descent to a quotient.  Specifically, we will define a map
\[ \widetilde \Phi \colon \C(S) \times \Diff_0(S) \to \sC \]
where $\Diff_0(S)$ is the component of the group of diffeomorphisms that of $S$ containing the identity (all diffeomorphisms of $S$ are assumed to extend to
to diffeomorphisms of the closed surface obtained by filling in the punctures).
To define $\widetilde \Phi$, first for each curve $\alpha \in \C_0(S)$, we let $\alpha$ denote the geodesic representative in our fixed hyperbolic metric on $S$, and choose once and for all $\epsilon(\alpha) >0$ so that for any two vertices $\alpha,\alpha'$, $i(\alpha,\alpha')$ is equal to the number of components of $N_{\epsilon(\alpha)}(\alpha) \cap N_{\epsilon(\alpha')}(\alpha')$.  
If $f(z)$ is disjoint from the interior of $N_{\epsilon(\alpha)}(\alpha)$, then $\widetilde \Phi(\alpha,f) = f^{-1}(\alpha)$, viewed as a curve on $\dot S$.  If $f(z)$ is contained in the interior of $N_{\epsilon(\alpha)}(\alpha)$, then we let $\alpha_\pm$ denote the two boundary components of this neighborhood, and define $\widetilde \Phi(\alpha,f)$ to be a point of the edge between the curves $f^{-1}(\alpha_-)$ and $f^{-1}(\alpha_+)$ determined by the relative distance to $\alpha_+$ and $\alpha_-$.
For a point $s \in \C(S)$ inside a simplex $\Delta$ of dimension greater than $0$, we use the neighborhoods as well as the barycentric coordinates of $s$ inside $\Delta$ to define $\widetilde \Phi(s,f) \in \Pi^{-1}(s) \subset \Pi^{-1}(\Delta)$; see \cite[Section~2.2]{LeinMjSch} for details.

Next we note that the isotopy from the identity to $f$ lifts to an isotopy from the identity to a {\em canonical lift} $\widetilde f$ of $f$. The map $\Phi$ is then defined from our choice $\widetilde z \in p^{-1}(z)$ and the canonical lift by the equation
\[ \Phi(\alpha,\widetilde f(\widetilde z)) = \widetilde \Phi(\alpha,f).\]
Alternatively, the we have the evaluation map $\ev \colon \Diff_0(S) \to S$, $\ev(f) = f(z)$, which lifts to a map $\widetilde{\ev} \colon \Diff_0(S) \to \mathbb H$ (given by $\widetilde{\ev}(f)= \widetilde f(\widetilde z)$, where again $\widetilde f$ is the canonical lift), and then $\Phi$ is defined as the descent by $\id_{\C(S)} \times \widetilde{\ev}$:
\[ \xymatrixcolsep{4pc}\xymatrixrowsep{2pc}\xymatrix{
\C(S) \times \Diff_0(S) \ar[rd]^{\widetilde \Phi} \ar[d]_{\id_{\C(S)} \times \widetilde{\ev}} \\
\C(S) \times \HH \ar[r]_\Phi & \C^s(\dot S).}\]
Note that every $\widetilde w \in \mathbb H$ is $\widetilde w = \widetilde f(\widetilde z)$ for some $f \in \Diff_0(S)$ (indeed, $\widetilde{\ev}$ defines a locally trivial fiber bundle).  As is shown in \cite{LeinMjSch}, $\Phi(\alpha,\widetilde w)$ is well-defined independent of the choice of such a diffeomorphism $f$ with $\widetilde f(\widetilde z) = \widetilde w$ since any two differ by an isotopy fixing $z$, and $\Phi$ is $\pi_1S$--equivariant (where the points $\widetilde z$ is used to identify the fundamental group with the group of covering transformations).
It is straightforward to see that $\Phi(\alpha, \cdot)$ is constant on components of $\mathbb H \setminus p^{-1}(N_{\epsilon(\alpha)}(\alpha))$: two points $\widetilde w,\widetilde w'$ in such a component are given by $\widetilde w = \widetilde f(\widetilde z)$ and $\widetilde w' = \widetilde f'(\widetilde z)$ where $f$ and $f'$ are isotopic by an isotopy $f_t$, so that $f_t(z)$ remains outside $N_{\epsilon(\alpha)}(\alpha)$ for all $t$.



\subsection{Cusps and witnesses} \label{S:cusps and witnesses}
The following lemma relates $\Phi$ to the proper witnesses.  Let $\mathcal P \subset \partial \mathbb H$ denote the set of parabolic fixed points.  Assume that for each $x \in \mathcal P$, we choose a horoball $H_x$ invariant by the parabolic subgroup $\stab_{\pi_1S}(x)$, the stabilizer of $x$ in $\pi_1S$.  We further assume, as we may, that (1) the union of the horoballs is $\pi_1S$--invariant, (2) the horoballs are pairwise disjoint (so all projected to pairwise disjoint cusp neighborhoods of the punctures), and (3) the horoballs all project disjoint from $N_{\epsilon(\alpha)}(\alpha)$ for all curves $\alpha$.  Recall that any proper witness is determined by its boundary.

\begin{lemma} \label{L:W parabolics}  There is a $\pi_1S$--equivariant bijection $\mathcal W \colon \mathcal P \to \Omega_0(\dot S)$ determined by
\begin{equation} \label{E:definition of W} \partial \mathcal W(x) = f^{-1}(\partial p(H_x)),
\end{equation}
for any $f \in \Diff_0(S)$ with $\widetilde f(\widetilde z)$ in the interior of the horoball $H_x$.
Moreover, $\Phi(\C(S) \times H_x) = \C(\mathcal W(x))$, we have $\Phi(\Pi(u) \times H_x)) = u$ for all $u \in \C(\mathcal W(x))$, and
$Stab_{\pi_1S}(x)$, acts trivially on $\C(\mathcal W(x))$.
\end{lemma}
From the lemma (and as illustrated in the proof) $\Phi|_{C(S) \times H_x}$ defines an isomorphism $\C(S) \to \C(\mathcal W(x))$ inverting the isomorphism $\Pi|_{\C(\mathcal W(x))} \colon \C(\mathcal W(x)) \to \C(S)$.
\begin{proof} For any $f \in \Diff_0(S)$ with $\widetilde w = \widetilde f(\widetilde z) \in H_x$ and any curve $v \in \C_0(S)$, we have $\Phi(v,\widetilde w) = \widetilde \Phi(v,f) = f^{-1}(v)$.  On the other hand, $f^{-1}(\partial p(H_x))$ is the boundary of a twice punctured disk containing the $z$ puncture, and hence $f^{-1}(\partial p(H_x))$ is the boundary of a witness we denote $\mathcal W(x)$.  Since $v$ and $\partial p(H_x)$ are disjoint,
\[ \Phi(v,\widetilde w) \in \C(\mathcal W(x)) \subset \sC.\]
The same proof that $\Phi(v,\widetilde w)$ is well-defined (independent of the choice of $f \in \Diff_0(S)$ with $\widetilde f(\widetilde z) = \widetilde w$), shows $f^{-1}(\partial p(H_x))$ is independent of such a choice of $f$ (up to isotopy).  Therefore, $\mathcal W$ is well defined by \eqref{E:definition of W}.  Since $v \in \C_0(S)$ was arbitrary and $\Phi(v,\cdot)$ is constant on components of the complement of $p^{-1}(N_{\epsilon(v)}(v))$, we have
\[ \Phi(\C(S) \times  H_x) \subset \C(\mathcal W(x)).\]
Given $u \in \C(\mathcal W(x))$, we view $u$ as a curve disjoint from $f^{-1}(\partial p(H_x))$ and hence $f(u)$ is disjoint from $p(H_x)$.  There is an isotopy of $f(u)$ to $v$ fixing $p(H_x)$ (since this is just a neighborhood of the cusp) and hence an isotopy of $u$ to $f^{-1}(v)$ disjoint from $f^{-1}(\partial p(H_x))$.   This implies $\Phi(\{v\} \times H_x) = u$, proving that $\Phi(\C(S) \times H_x) = \C(\mathcal W(x))$, as well as the formula $\Phi(\Pi(u)\times H_x) = u$ for all $u \in \C(\mathcal W(x))$.

Next observe that for any proper witness $W$, the subcomplex $\C(W) \subset \C^s(\dot S)$ uniquely determines $W$.  Therefore, the property that $\Phi(\C(S) \times H_x) = \C(\calW(x))$, together with the $\pi_1S$--equivariance of $\Phi$ implies that $\calW$ is $\pi_1S$--equivariant.  All that remains is to show that $\calW$ is a bijection.  Let $C_1,\ldots,C_n$ be the pairwise disjoint horoball cusp neighborhoods of the punctures obtained by projecting the horoballs $H_x$ for all $x \in \calP$.

For any proper witness $W$, there is a diffeomorphism $f \colon S \to S$,  isotopic to the identity by an isotopy $f_t$ which is the identity on $W$ for all $t$, and so that $f(z) \in C_i$, for some $i$.  Note that there is an arc connecting $z$ to the $i^{th}$ puncture which is disjoint from both $\partial W$ and $\partial C_i$.  It follows that $\partial W$ and $\partial C_i$ are isotopic, and thus by further isotopy (no longer the identity on $W$) we may assume that $f(\partial W) = \partial C_i$.  Therefore, $f^{-1}(\partial C_i) = \partial W$.  Observe that the canonical lift $\tilde f$ has $\tilde f(\tilde z) \in H_x$ for some $x \in \calP$ with $p(H_x) = C_i$.  Therefore, $f^{-1}(\partial p(H_x)) = W$, and so $\mathcal W(x) = W$, so $\calW$ is surjective.

To see that $\calW$ is injective, suppose $x,y \in \calP$ are such that $\calW(x) = \calW(y)$.  The two punctures surrounded by $\partial \calW(x)$ and by $\partial \calW(y)$ are therefore the same, hence there exists an element $\gamma \in \pi_1S$ so that $\gamma \cdot x= y$.  By $\pi_1S$--equivariance, we must have
\[ \gamma \cdot \partial \calW(x) = \partial \calW(\gamma \cdot x) = \partial \calW(y) = \partial \calW(x).\]
Choose a representative loop for $\gamma$ with minimal self-intersection and denote this $\gamma_0$.
If $\gamma_0$ is simple closed, then the mapping class associated to $\gamma$ is the product of Dehn twists (with opposite signs) in the boundary curves of a regular neighborhood of $\gamma_0$. Otherwise, $\gamma_0$ fills a subsurface $Y \subset \dot S$ and is pseudo-Anosov on this subsurface by a result of Kra \cite{Kra} (see also \cite{LeinKentSch}).  It follows that $\gamma \cdot \partial \calW(x) = \partial \calW(x)$ if and only if $\gamma_0$ is disjoint from $\partial \calW(x)$, which happens if and only if $f(\gamma_0) \subset p(H_x)$ (up to isotopy relative to $f(z)$).  In the action of $\pi_1S$ on $\HH$, the element $\gamma$ sends $\tilde f(\tilde z)$ to $\gamma \cdot \tilde f(\tilde z)$, and these are the initial and terminal endpoints of the $\tilde f$--image of the lift of $\gamma_0$ with initial point $\tilde z$.
On the other hand, $\tilde f(\tilde z) \in H_x$, and hence so is $\gamma  \cdot \tilde f(\tilde z)$, which means that $\gamma$ is fixes $x$.  Therefore, $y = \gamma \cdot x = x$, and thus $\calW$ is injective.
\end{proof}

\subsection{Spaces of laminations} \label{S:laminations}
We refer the reader to \cite{Thurston1}, \cite{CEG}, \cite{FLP}, and \cite{CB} for details about the topics discussed here.  By a {\em lamination} on a surface $Y$ we mean a {\em compact} subset of the interior of $Y$ foliated by complete geodesics with respect to some complete, hyperbolic metric of finite area, with (possibly empty) geodesic boundary; the geodesics in the foliation are uniquely determined by the lamination and are called the {\em leaves}.  For example, any simple closed geodesic $\alpha$ is a lamination with exactly one leaf.  For a fixed complete, finite area, hyperbolic metric on $Y$, all geodesic laminations are all contained in a compact subset of the interior of $Y$.  For any two complete, hyperbolic metrics of finite area, laminations that are geodesic with respect to the first are isotopic to laminations that are geodesic with respect to the second.  In fact, we can remove any geodesic boundary components, and replace the resulting ends with cusps, and this remains true.  We therefore sometimes view laminations as well-defined up to isotopy, unless a hyperbolic metric is specified in which case we assume they are geodesic.

A {\em complementary region} of a lamination $\calL \subset Y$ is the image in $Y$ of the closure of a component of the preimage in the universal covering; intuitively, it is the union of a complementary component together with the ``leaves bounding this component''.  We view the complementary regions as immersed subsurfaces with (not necessarily compact) boundary consisting of arcs and circles (for a generic lamination, the immersion is injective, though in general it is only injective on the interior of the subsurface).  We will also refer to the closure of a complementary component in the universal cover of $Y$ as a complementary region (of the preimage of a lamination).

We write $\calG\calL(Y)$ for the set of laminations on the surface $Y$, dropping the reference to $Y$ when it is clear from the context.  The set of essential simple closed curves, up to isotopy (i.e.~the vertex set of $\C(Y)$) is thus naturally a subset of $\calG\calL(Y)$. A lamination is {\em minimal} if every leaf is dense in it, and it is \textit{filling} if its complementary regions are ideal polygons, or one-holed ideal polygons where the hole is either a boundary component or cusp of $Y$.  A sublamination of a lamination is a subset which is also a lamination.  Every lamination decomposes as a finite disjoint union of simple closed curves, minimal sublaminations without closed leaves (called the {\em minimal components}), and biinfinite isolated leaves (leaves with a neighborhood disjoint from the rest of the lamination).

There are several topologies on $\calG\calL$ that will be important for us (in what follows, and whenever discussing convergence in the topologies, we view laminations as geodesic laminations with respect to a fixed complete hyperbolic metric of finite area; the resulting topology and convergence is independent of the choice of metric).  The first is a metric topology called the {\em Hausdorff topology} (also known as the {\em Chabauty topology}), induced by the {\em Hausdorff metric} on the set of all compact subsets of a compact space (in our case, the compact subset of the surface that contains all geodesic laminations) defined by
\[ d_H(\calL,\calL') = \inf \{\epsilon >0 \mid \calL \subset \calN_{\epsilon}(\calL') \mbox{ and } \calL' \subset \calN_{\epsilon}(\calL) \}.\]
If a sequence of $\{\calL_i\}$ converges to $\calL$ in this topology, we write $\calL_i \xrightarrow{\text{H}} \calL$.  The following provides a useful characterization of convergence in this topology; see \cite{CEG}.

\begin{lemma} \label{L:Hausdorff convergence} We have $\calL_i \xrightarrow{\text{H}} \calL$  if and only if
\begin{enumerate}
  \item for all $x \in \calL$ there is a sequence of points $x_i \in \calL_i$ so that $x_i \to x$, and
  \item for every subsequence $\{\calL_{i_k}\}_{k=1}^\infty$, if $x_{i_k} \in \calL_{i_k}$, and $x_{i_k} \to x$, then $x \in \calL$.
\end{enumerate}
\end{lemma}

This lemma holds not just for Hausdorff convergence of laminations, but for any sequence of compact subsets of a compact metric space with respect to the Hausdorff metric.

The set $\calG\calL$ can also be equipped with a weaker topology called the \emph{coarse Hausdorff topology}, \cite{Ham-CH}, introduced by Thurston in \cite{TNotes} where it was called the {\em geometric topology} (see also \cite{CEG} where it was {\em Thurston topology}).  If a sequence $\{\calL_i\}$ converges to $\calL$ in the coarse Hausdorff topology, then we write $\calL_i \xrightarrow{\text{CH}} \calL$.   The following describes convergence in this topology; see \cite{CEG}.
 \begin{lemma} \label{L:CH} We have $\calL_i \xrightarrow{\text{CH}} \calL$ if and only if condition (1) holds from Lemma~\ref{L:Hausdorff convergence}.
\end{lemma}

The next corollary gives a useful way of understanding coarse Hausdorff convergence.

\begin{corollary} \label{C:CH sublamination} We have $\calL_i \xrightarrow{\text{CH}} \calL$ if and only if every Hausdorff convergent subsequence converges to a lamination $\calL'$ containing $\calL$.
\end{corollary}

Since any lamination has only finitely many sublaminations, from the corollary we see that while limits are not necessarily unique in the coarse Hausdorff topology, a sequence can have only finitely many limits.  We let $\EL=\EL(Y)$ denote the space of {\em ending laminations} on $Y$, which are minimal, filling laminations, equipped with the coarse Hausdorff topology.  As suggested by the name, these are precisely the laminations that occur as the ending laminations of a type preserving, proper, isometric action on hyperbolic $3$--space without accidental parabolics as discussed in the introduction.

A \textit{measured lamination} is a lamination $\calL$ together with an invariant transverse measure $\mu$; that is, an assignment of a measure on all arcs transverse to the lamination, satisfying natural subdivision properties which is invariant under isotopy of arcs preserving transversality with the lamination. The support of a measured lamination $(\calL,\mu)$ is the sublamination $|\mu| \subseteq \calL$ with the property that a transverse arc has positive measure if and only if the intersection with $|\mu|$ is nonempty, and is a union of minimal components and simple closed geodesics.  We often assume that $(\calL,\mu)$ has {\em full support}, meaning $\calL = |\mu|$.  In this case, we sometimes write $\mu$ instead of $(\mathcal L,\mu)$.

The space $\ML = \ML(Y)$ of {\em measured laminations} on $Y$ is the set of all measured laminations of full support equipped with the weak* topology on measures on an appropriate family of arcs transverse to all laminations.  Given an arbitrary measured lamination, $(\calL,\mu)$, we have $(|\mu|,\mu)$ is an element of $\ML$, and so every measured lamination determines a unique point of $\ML$.   We let $\FL \subset \ML$ denote the subspace of measured laminations whose support is an ending lamination (i.e.~it is in $\EL$).  We write $\PML$ and $\PFL$ for the respective projectivizations of $\ML$ and $\FL$, obtained by taking the quotient by scaling measures, with the quotient topologies.  The following will be useful in the sequel; see \cite[Chapter 8.10]{TNotes}.

\begin{proposition} \label{P:support continuous} The map $\PML \to \calG\calL$, given by $\mu \mapsto |\mu|$, is continuous with respect to the coarse Hausdorff topology on $\calG\calL$.
\end{proposition}


For the surface $\dot S$, we consider the subspace
\[\sEL:=\bigsqcup_{W \in \Omega(\dot S)} \calE \calL(W) \subset {\calG\calL}(\dot S),\]
which is the union of ending laminations of all witnesses of $\dot S$.
Similarly, we will write $\FL^s(\dot S) \subset \ML(\dot S)$ for those measured laminations supported on laminations in $\sEL$, and $\PFL^s(\dot S) \subset \PML(\dot S)$ for its projectivization.

\subsection{Gromov Boundary of  a hyperbolic space}

A $\delta$--hyperbolic space $\calX$ can be equipped with a boundary at infinity, $\partial \calX$ as follows.  Given $x,y \in \calX$ and a basepoint $o \in \calX$, the Gromov product of $x$ and $y$ based at $o$ is given by
\[ \langle x, y \rangle_o = \frac12\left( d(x,o)+d(y,o) - d(x,y) \right).\]
Up to a bounded error (depending only on $\delta$), $\langle x,y \rangle_o$ is the distance from $o$ to a geodesic connecting $x$ and $y$.  The quantity $\langle x,y \rangle_o$ is estimated by the distance from the basepoint $o$ to a quasi-geodesic between $x$ and $y$.  There is an additive and multiplicative error in the estimate that depends only on the hyperbolicity constant and the quasi-geodesic constants.  Using and slim triangles, we also note that for all $x,y,z \in \calX$,
\[ \langle x,y \rangle_o \succeq \min\{ \langle x,z \rangle_o, \langle y,z \rangle_o\} \]
where the constants in the coarse lower bound depend only on the hyperbolicity constant.

A sequence $\{x_n\} \subset \calX$ is said to {\em converge to infinity} if $\displaystyle{\lim_{m,n\rightarrow\infty}\la x_m,x_n\ra_o=\infty}$.   Two sequences $\{x_n\}$ and $\{y_n\}$ are equivalent if $\displaystyle{\lim_{m,n\rightarrow\infty}\la y_m,x_n \ra =\infty}$.  The points in $\partial \calX$ are equivalence classes of sequences converging to infinity, and if $\{x_k\} \in  x \in \partial \calX$, then we say $\{x_k\}$ converges to $x$ and write $x_k \to x$ in $\overline{\calX} = \calX \cup \partial \calX$.  The topology on the boundary is such that a sequence $\{x^n\}_n \subset \partial X$ converges to a point $x \in \partial X$ if there exist sequences $\{x_k^n\}_k$ representing $x^n$ for all $n$, and $\{x_m\}_m$ representing $x$ so that
\[ \lim_{n\to \infty} \liminf_{k,m \to \infty} \langle x_k^n,x_m\rangle_o = \infty.\]
For details see, e.g.~\cite{BH,Kapovich-bdry}.

Klarreich \cite{Klarreich} proved that the Gromov boundary of the curve complex is naturally homeomorphic to the space of ending laminations equipped with the quotient topology from $\FL \subset \ML$ using the geometry of the Teichmuller space\footnote{In fact, Klarreich worked with the space of measured foliations, an alter ego of the space of measured laminations.}.  Hamenst\"adt \cite{Ham-CH} gave a new proof, endowing $\EL$ with the coarse Hausdorff topology (which for $\EL$ is the same topology as the quotient topology), also providing additional information about convergence.  Yet another proof of the version we use here was given by Pho-On \cite{Bom}.

\begin{theorem}  \label{T:Klarreich} For any surface $Y$ equipped with a complete hyperbolic metric of finite area (possibly having geodesic boundary), there is a homeomorphism $\mathcal F_Y \colon \partial \C(Y) \to \EL(Y)$ so that $\alpha_n\rightarrow x$ if and only if $\alpha_n \xrightarrow{\text{CH}} \calF_Y(x)$. 
\end{theorem}

\subsection{Laminations and subsurfaces}  The following lemma relates coarse Hausdorff convergence of a sequence to coarse Hausdorff convergence of its projection to witnesses in important special case.
\begin{lemma} \label{L:proj and conv} If $\{\alpha_n\} \subset \C_0^s(\dot S)$ and $\calL \in \EL(W)$ for some witness $W$, then $\alpha_n \stackrel{CH}{\to} \calL$ if and only if $\pi_W(\alpha_n) \stackrel{CH}{\to} \calL$.
\end{lemma}
Note that for each $n$, $\pi_W(\alpha_n)$ is a union of curves, which are not necessarily disjoint.  In particular, $\pi_W(\alpha_n)$ is not necessarily a geodesic laminations, so we should be careful in discussing its coarse Hausdorff convergence.  However, viewing the union as a subset of $\C(W)$, it has diameter at most $2$, and hence if $a_n,a_n' \subset \pi_W(\alpha_n)$ are any two curves, for each $n$, and $\calL \in \EL(W)$, then $a_n$ and $a_n'$ either both coarse Hausdorff converge to $\calL$ or neither does (by Theorem~\ref{T:Klarreich}).  Consequently, it makes sense to say that $\pi_W(\alpha_n)$ coarse Hausdorff converges to a lamination in $\EL(W)$.

\begin{proof}  For the rest of this proof we fix a complete hyperbolic metric on $\dot S$ and realize $W \subset \dot S$ as an embedded subsurface with geodesic boundary.   Let us first assume  $\pi_W(\alpha_n) \stackrel{CH}{\to} \calL \in \EL(W)$.  After passing to an arbitrary convergent subsequence, we may assume $\alpha_n \stackrel{H}{\to} {\calL'}$.  It suffices to show that $\calL \subset \calL'$.

Let $\ell_n^1,\ldots,\ell_n^r \subset \alpha_n \cap W$ be the decomposition into isotopy classes of arcs of intersection: that is, each $\ell_n^j$ is a union of all arcs of intersection of $\alpha_n$ with $W$ so that any two arcs of $\alpha_n \cap W$ are isotopic if and only if they are contained in the same set $\ell_n^j$  (we may have to pass to a further subsequence so that each intersection $\alpha_n \cap W$ consists of the same number $r$ of isotopy classes, which we do).  For each $\ell_n^j$, let $\alpha_n^j \subset \pi_W(\alpha_n)$ be the geodesic multi-curve produced from the isotopy class $\ell_n^j$ by surgery in the definition of projection.  Note that $\alpha_n^j$ and $\ell_n^j$ have no transverse intersections.  Pass to a further subsequence so that $\alpha_n^j \stackrel{H}\to \calL^j$ and $\ell_n^j \stackrel{H}\to \calL_j'$; here, each $\ell_n^j$ is a compact subset of $W$ so Hausdorff convergence to a closed set still makes sense, though $\calL_j'$ are not necessarily geodesic laminations.  By Corollary~\ref{C:CH sublamination} (and the discussion in the paragraph preceding this proof), $\calL \subset \calL^j$, for each $j$.   Appealing to Lemma~\ref{L:Hausdorff convergence}, it easily follows that  $\calL' \cap W = \calL_1' \cup \cdots \cup \calL_r'$.  Since $a_n^j$ has no transverse intersections with $\ell_n^j$, $\calL_j'$ has no transverse intersections with $\calL^j$, for each $j$.  Therefore, $\calL$ has no transverse intersections with $\calL' \cap W$, and since $\calL \subset W$, $\calL'$ has no transverse intersections with $\calL$.  Since $\calL \in \EL(W)$, it follows that $\calL \subset \calL'$, as required.

Now in the opposite direction we assume that  $\alpha_n \stackrel{CH}{\to} \calL \in \EL(W)$.  Let $\ell_n^1,\ldots,\ell_n^r \subset \alpha_n \cap W$ and $\alpha_n^1,\ldots,\alpha_n^r \subset \pi_W(\alpha_n)$ be as above, so that for each $j$ (after passing to a subsequence) we have
\[ \ell_n^j \stackrel{H}\to \calL^j \quad \mbox{ and } \quad \alpha_n^j \stackrel{H} \to \calL_j'. \]
Similar to the above, $\calL \subset \calL^1 \cup \cdots \cup \calL^r$ and since $\ell_n^j$ has no transverse intersections with $\alpha_n^j$, $\calL_j'$ has no transverse intersections with $\calL$.  Since $\calL$ is an ending lamination, $\calL \subset \calL_j'$.  Since the convergent subsequence was arbitrary, it follows that $\pi_W(\alpha_n) \stackrel{CH}\to \calL$.
\end{proof}

Finally, we note that just as curves can be projected to subsurfaces, whenever a lamination minimally intersects a subsurface in a disjoint union of arcs, we may use the same procedure to project laminations.

\section{Survival paths} \label{S:survival paths}

To understand the geometry of $\C^s(\dot S)$, the Gromov boundary, and the Cannon-Thurston map we eventually construct, we will make use of some special paths we call {\em survival paths}.  To describe their construction, we set the following notation.  Given a witness $W \subseteq \dot S$ and $x,y \in \C(W)$, let $[x,y]_W \subset \C(W)$ denote a geodesic between $x$ and $y$.

The following definition is reminiscent of hierarchy paths from \cite{MM2}, though our situation is considerably simpler.
\begin{definition}  Given $x,y \in \C^s(\dot S)$, let $[x,y]_{\dot S}$ be any $\C(\dot S)$--geodesic.  If $W \subsetneq \dot S$ is a proper witness such that $\partial W$ is a vertex of $[x,y]_{\dot S}$, we say that $W$ is a {\em witness for $[x,y]_{\dot S}$}.  Note that if $W$ is a witness for $[x,y]_{\dot S}$, then the immediate predecessor and successor $x',y'$ to $\partial W$  in $[x,y]_{\dot S}$ are necessarily contained in $\C(W)$ (hence also in $\C^s(\dot S)$) and we let $[x',y']_W \subset \C(W)$ be a geodesic (which we also view as a path in $\C^s(\dot S)$).  Replacing every consecutive triple $x',\partial W, y' \subset [x,y]_{\dot S}$ with the path $[x',y']_W$ produces a path from $x$ to $y$ in $\C^s(\dot S)$ which we call a {\em survival path} from $x$ to $y$, and denote it $\sigma(x,y)$.  We call $[x,y]_{\dot S}$ the {\em main geodesic} of $\sigma(x,y)$ and, if $W$ is witness for $[x,y]_{\dot S}$, we call the corresponding $\C(W)$--geodesic $[x',y']_W$ the {\em witness geodesic} of $\sigma(x,y)$ for $W$, and also say that $W$ is a {\em witness for $\sigma(x,y)$}.
\end{definition}

An immediate corollary of Theorem~\ref{BGIT}, we have
\begin{corollary} \label{C:necessarily witnesses} For any $x,y \in \C^s(\dot S)$ and proper witness $W$, if $d_W(x,y) > M$, then $W$ is a witness for $[x,y]_{\dot S}$, for any geodesics $[x,y]_{\dot S}$ between $x$ and $y$.
\end{corollary}
\begin{proof}  Since $d_W(x,y) > M$, it follows by Theorem~\ref{BGIT} that some vertex of $[x,y]_{\dot S}$ has empty projection to $W$.  But the only multi-curve in $\C(\dot S)$ with empty projection to $W$ is $\partial W$, hence $\partial W$ is a vertex of $[x,y]_{\dot S}$.
\end{proof}

No two consecutive vertices of $[x,y]_{\dot S}$ can be boundaries of witness (since any two such boundaries nontrivially intersect).  Therefore, the next lemma follows.
\begin{lemma} \label{L:distance witnesses}
For any $x,y \in \sC$ and geodesics $[x,y]_{\dot S}$, there are at most $\tfrac{d_{\dot S}(x,y)}2$ witnesses for $[x,y]_{\dot S}$. \qed
\end{lemma}

The following lemma estimates the lengths of witness geodesics on a survival path.
\begin{lemma}\label{L:witnessLength} Given a survival path $\sigma (x,y)$ and a witness $W$ for $\sigma(x,y)$, the initial and terminal vertices $x'$ and $y'$ of the witness geodesic segment $[x', y']_{W}$ satisfy
\[ d_W(x,x'),d_W(y,y') < M.\]
Consequently, $d_W(x',y')$ of  satisfies
\[ d_W(x',y') \stackrel{\mbox{\tiny $2M\! ,\! 0$}}{\asymp} d_W(x,y).\]
\end{lemma}
\begin{proof} By Theorem~\ref{BGIT} applied to the subsegments of $[x,y]_{\dot S}$ from $x$ to $x'$ and $y'$ to $y$ proves the first inequality.  The second is immediate from the triangle inequality.
\end{proof}

Finally we have the easy half of a distance estimate (c.f.~\cite{MM2}).
\begin{lemma} \label{L:upper bound formula}  For any $x,y \in \C^s(\dot S)$ and $k > M$, we have
\[ d^s(x,y) \leq 2k^2 + 2k \sum_{W \in \Omega(\dot S)} \lcut d_W(x,y) \rcut_k. \]
\end{lemma}
Recall that $\Omega(\dot S)$ denotes the set of all witnesses for $\C^s(\dot S)$ and that $\lcut x \rcut_k$ is the cut-off function giving value $x$ if $x \geq k$ and $0$ otherwise.

\begin{proof}  Since $\sigma(x,y)$ is a path from $x$ to $y$, it suffices to prove that the length of $\sigma(x,y)$ is bounded above by the right-hand side.  For each witness $W$ of $x,y$ whose boundary appears in $[x,y]_{\dot S}$, we have replaced the length two segment $\{x',\partial W, y'\}$ with $[x',y']_W$, which has length $d_W(x',y')$.  By Lemma \ref{L:witnessLength} we have

\[ d_W(x',y') \leq 2M + d_W(x,y).\]

If $d_W(x,y) \geq k > M$, this implies the length $d_W(x',y')$, of $[x',y']_W$ is less than $3d_W(x,y)$.  Otherwise, the length is less than $3k$.
Let $W_1,\ldots, W_n$ denote the witnesses for $x,y$ whose boundaries appear in $[x,y]_{\dot S}$.  By Lemma~\ref{L:distance witnesses}, $n \leq \tfrac12 d_{\dot S}(x,y)$, half the length of $[x,y]_{\dot S}$.  Further note that by Corollary~\ref{C:necessarily witnesses}, if $d_W(x,y) \geq k > M$, then $W$ is one of the witnesses $W_j$, for some $j$.

Combining all of these (and fact that $k > M > 2$) we obtain the following bound on the length of $\sigma(x,y)$, and hence $d^s(x,y)$:
\[ \begin{array}{rclcl}
d^s(x,y) & \leq & \displaystyle{d_{\dot S}(x,y)  + \sum_{j=1}^n 3d_{W_j}(x,y)} \leq  \displaystyle{d_{\dot S}(x,y)+ 3 \sum_{j=1}^n (\lcut d_{W_j}(x,y) \rcut_k + k)}\\
& = & \displaystyle{d_{\dot S}(x,y) + 3nk + 3 \sum_{j=1}^n \lcut d_{W_j}(x,y) \rcut_k} \leq \displaystyle{(\tfrac{3k}2 + 1)d_{\dot S}(x,y) + 3 \sum_{j=1}^n \lcut d_{W_j}(x,y) \rcut_k}\\
& \leq & \displaystyle{2k(\lcut d_{\dot S}(x,y) \rcut_k + k) + 3 \sum_{j=1}^n \lcut d_{W_j}(x,y) \rcut_k} \leq \displaystyle{2k^2 + 2k \sum_{W \in \Omega(\dot S)} \lcut d_W(x,y) \rcut_k} \\
\end{array} \]
\end{proof}

\begin{lemma} \label{L:projecting survival paths} Given $x,y \in \sC$, if $W$ is not a witness for $[x,y]_{\dot S}$, then
\[ \diam_W(\sigma(x,y)) \leq M + 4.\]
\end{lemma}
\begin{proof} Since $W$ is not a witness for $[x,y]_{\dot S}$, every $z \in [x,y]_{\dot S}$ has non-empty projection to $W$.  Therefore, $\diam_W([x,y])_{\dot S} \leq M$ by Theorem~\ref{BGIT}.  If $w' \in \C(W')$ is on a witness geodesic segment of $\sigma(x,y)$, then $d_{\dot S}(w',\partial W') =1$ so $d_W(w',\partial W') \leq 2$ by Proposition~\ref{P:2 Lipschitz}.  Since $\partial W' \in [x,y]_{\dot S}$, the lemma follows by the triangle inequality.
\end{proof}


\begin{lemma} \label{L:subpaths of survival paths} Suppose $\sigma(x,y)$ is a survival path and $x',y' \in \sigma(x,y)$ with $x \leq x' <y' \leq y$, with respect to the ordering from $\sigma(x,y)$.  Then if $x',y'$ lie on the main geodesic $[x,y]_{\dot S}$, then the subpath of $\sigma(x,y)$ from $x'$ to $y'$ is a survival path.

If $x' \in \C(W)$ and/or $y' \in \C(W')$ for proper witnesses $W,W'$ for $x,y$, respectively, then the same conclusion holds, provided the subsegments of $\C(W)$ and/or $\C(W')$ in $\sigma(x,y)$ between $x'$ and $y'$ has length at least $2M$.
\end{lemma}
\begin{proof} When $x',y'$ are on the main geodesic, this is straightforward, since in this case, the subsegment of the main geodesic between $x'$ and $y'$ serves as the main geodesic for a survival path between $x'$ and $y'$.

There are several cases for the second statement.  The proofs are all similar, so we just describe one case where, say, $x' \in [x'',y'']_W \subset \C(W)$ with $x \leq x'' \leq x' \leq y'' \leq y' \leq y$, and $y'$ is in the main geodesic.  The assumption in this case means that in $\C(W)$, the distance between $x'$ and $y''$ is at least $2M$.  Lemma~\ref{L:witnessLength} implies that $d_W(y'',y) < M$, and so by the triangle inequality, $d_W(x',y) > M$.  Therefore, by Theorem~\ref{BGIT} any geodesic from $x'$ to $y$ must pass through $\partial W$.  In particular, the path that starts at $x'$, travels to $\partial W$, then continues along the subsegment of $[x,y]_{\dot S}$ from $\partial W$ to $y'$, is a geodesic in $\C(\dot S)$.   We can easily build a survival path from $x'$ to $y'$ using this geodesic that is a subsegment of $\sigma(x,y)$, as required.  The other cases are similar.
\end{proof}

\subsection{Infinite survival paths}

Masur-Minsky proved that for any surface $Z$ and any two points in $\bar \C(Z) = \C(Z) \cup \partial \C(Z)$, where $\partial \C(Z)$ is the Gromov boundary, there is a geodesic ``connecting" these points; see \cite{MM2}.  Given $x,y \in \bar \C(Z)$, we let $[x,y]_Z$ denote such a geodesic.

The construction of survival paths above can be carried out for geodesic lines and rays in $\C(\dot S)$, replacing any length two path $x',\partial W, y'$ with a $\C(W)$ geodesic from $x'$ to $y'$ to produce a {\em survival ray} or {\em survival line}, respectively.  More generally, to a geodesic segment or ray of $\C(\dot S)$ we can construct other types of survival rays and survival lines.  Specifically, first construct a survival path as above or as just described, then append to one or both endpoints an infinite witness ray (or rays).  For example, for any two distinct witnesses $W$ and $W'$ and points $z,z'$ in the Gromov boundaries of $\C(W)$ and $\C(W')$, respectively, we can construct a survival line starts and ends with geodesic rays in $\C(W)$ and $\C(W')$, limiting to $z$ and $z'$, respectively, and having main geodesic being a segment.  In this way, we see that survival lines can thus be constructed for any pair of distinct points in
\[ z,z' \in \bigcup_{W \in \Omega(\dot S)} \bar \C(W),\]
and we denote such by $\sigma(z,z')$, as in the finite case.  From this discussion, we have the following.

\begin{lemma} \label{L:visual on witness boundaries} For any distinct pair of elements
\[ z,z' \in \bigcup_{W \in \Omega(\dot S)} \bar \C(W) \]
there exists a (possibly infinite) survival path $\sigma(z,z')$ ``connecting" these points.\qed
\end{lemma}

The next proposition allows us to deduce many of the properties of survival paths to infinite survival paths.
\begin{proposition} \label{P:exhaustion by survival} Any infinite survival path (line or ray) is an increasing union of {\em finite} survival paths.
\end{proposition}
\begin{proof} This follows just as in the proof of Lemma~\ref{L:subpaths of survival paths}.
\end{proof}

\begin{remark} Unless otherwise stated, the term ``survival path" will be reserved for finite survival paths.   ``Infinite survival path" will mean either survival ray or survival line.
\end{remark}

\section{Hyperbolicity of the surviving curve complex} \label{S:hyperbolicity}

In this section we prove the following theorem using survival paths.   The proof appeals to Proposition~\ref{bow}, due to Masur-Schleimer \cite{MSch1} and Bowditch (\cite{Bowhyp2}), which gives criteria for hyperbolicity.
\begin{theorem} \label{hyp}
The complex $\sC$ is Gromov-hyperbolic.
\end{theorem}
\begin{remark} There are alternate approaches to proving Theorem~\ref{hyp}.  For example, Masur and Schleimer provide a collection of axioms in  \cite{MSch1} whose verification would imply hyperbolicity.  Another approach would be to show that Vokes' condition for hyperbolicity in \cite{Vokes} which requires an action of the entire mapping class group can be relaxed to requiring an action of the stabilizer of $z$, which is a finite index subgroup of the mapping class group.  We have chosen to give a direct proof using survival paths since it is elementary and illustrates their utility.
\end{remark}

The condition for hyperbolicity we  use is the following; see \cite{MSch1,Bowhyp2}.

\begin{proposition}\label{bow}  Given $\epsilon>0$, there exists $\delta>0$ with the following property. Suppose that $G$ is a connected graph and for each $x,y \in V(G)$ there is an associated connected subgraph $\varsigma(x,y)\subseteq G$ including $x,y$. Suppose that,
\begin{enumerate}
  \item For all $x,y,z \in V(G)$,
  \[\varsigma(x,y)\subseteq \calN_{\epsilon}(\varsigma(x,z)\cup \varsigma(z,y))\]
  \item For any $x,y \in V(G)$ with $d(x,y)\leq 1$, the diameter of $\varsigma (x,y)$ in $G$ is at most $\epsilon$.
\end{enumerate}
Then, $G$ is $\delta$--hyperbolic.
\end{proposition}

We will apply Proposition \ref{bow} to the graph $\C^s(\dot S)$, and for vertices $x,y \in \C^s(\dot S)$, the required subcomplex is a (choice of some) survival path $\sigma(x,y)$.  Note that if $x,y$ are distance one apart, then $\sigma(x,y) = [x,y]$, which has diameter $1$.  Therefore, as long as $\epsilon \geq 1$, condition (2) in Theorem~\ref{bow} will be satisfied.  We therefore focus on condition (1), and express this briefly by saying that $x,y,z$ span an {\em $\epsilon$--slim survival triangle}.  The next lemma verifies condition (1) in a special case.

\begin{lemma} \label{L:cobounded slim triangles} Given $R > 4$, there exists $\epsilon > 0$ with the following property.  If $x,y,z \in \C^s(\dot S)$ are any three points such that $d_W(u,v) \leq R$ for all proper witness $W \subsetneq \dot S$ and every $u,v \in \{x,y,z\}$, then $x,y,z$ span an $\epsilon$--slim survival triangle.
\end{lemma}
\begin{proof} First note that by Lemma~\ref{L:witnessLength}, the length of any witness geodesic of any one of the three sides is at most $R + 2M$; we will use this fact throughout the proof without further mention.  We also observe that by Theorem~\ref{BGIT}, for any $w \in \sigma(x,y) \cap [x,y]_{\dot S}$ and any proper witness $W \subsetneq \dot S$, at least one of $d_W(x,w)$ or $d_W(y,w)$ is at most $M$.

Next suppose $w$ is on a subsegment $[x',y']_W \subset \sigma(x,y)$ for some proper witness $W$ of $\sigma(x,y)$.  Observe that $w$ is within $\tfrac{R+2M}2 = \tfrac{R}2 + M$ from either $x'$ or $y'$ and so by Theorem~\ref{BGIT} and the triangle inequality, one of $d_W(w,x)$ or $d_W(w,y)$ is at most $\tfrac{R}2+2M$.  If $W'$ is any other proper witness, we claim that $d_{W'}(x,w)$ or $d_{W'}(y,w)$ is at most $M+2 \leq \frac{R}2+2M$.   To see this, note that either $\partial W'$ lies in $[x,\partial W]_{\dot S} \subset [x,y]_{\dot S}$, in $[\partial W,y]_{\dot S} \subset [x,y]_{\dot S}$, or neither.  In the first two cases,  $d_{W'}(\partial W,y) \leq M$ or $d_{W'}(\partial W,x) \leq M$, respectively, by Theorem~\ref{BGIT}, while in the third case both of these inequalities hold.  Therefore, since $w$ and $\partial W$ are disjoint, $d_{W'}(\partial W,w) \leq 2$, and hence $d_{W'}(x,w)$ or $d_{W'}(y,w)$ is at most $M+2 \leq \frac{R}2+2M$.

Now let $w \in \sigma(x,y)$ be any vertex and $w_0 \in [x,y]_{\dot S} \cap \sigma(x,y)$ the nearest vertex along $\sigma(x,y)$, and observe that $d_{\dot S}(w,w_0) \leq 2$.   Since $\C(\dot S)$ is $\delta$--hyperbolic (for some $\delta > 0$), there is a vertex $w_0' \in [x,z]_{\dot S} \cup [y,z]_{\dot S}$ with $d_{\dot S}(w_0,w_0') \leq \delta$.  Without loss of generality, we assume $w_0' \in [x,z]_{\dot S}$.  Choose $w' \in \sigma(x,z)$ to be $w' = w_0'$ if $w_0' \in \sigma(x,z)$ or one of the adjacent vertices of $[x,z]_{\dot S}$ if $w_0'$ is the boundary of a witness.  Then  $d_{\dot S}(w_0',w') \leq 1$, so
\[ d_{\dot S}(w,w') \leq \delta + 3. \]

Now suppose $W \subsetneq \dot S$ is a proper witness.  Then at least one of $d_W(w,x)$ or $d_W(w,y)$ is at most $\tfrac{R}2 + 2M$ as is at least one of $d_W(w',x)$ or $d_W(w',z)$.   If $d_W(w,x),d_W(w',x) \leq \frac{R}2+2M$, then applying the triangle inequality, we see that
\[ d_W(w,w') \leq R + 4M.\]
If instead, $d_W(x,w) \leq \frac{R}2 + 2M$ and $d_W(w',z) \leq \frac{R}2 + 2M$, then the triangle inequality implies
\[  d_W(w,w') \leq d_W(w,x) + d_W(x,z) + d_W(z,w') \leq 2R + 4M.\]
The other two possibilities are similar, and hence $d_W(w,w') \leq 2R + 4M$.

Applying Corollary~\ref{L:upper bound formula} with $k = M$, recalling that by Lemma~\ref{L:distance witnesses} there are at most $\tfrac{d_{\dot S}(w,w')}2 \leq \frac{\delta + 3}2$ proper witnesses for any  geodesic $[w,w']_{\dot S}$, we have
\[ d^s(w,w') \leq 2M^2 + 2M \sum_{W \in \Omega(\dot S)} \lcut d_W(w,w') \rcut_M \leq 2M^2+2M \left( \delta + 3 + \tfrac{\delta + 3}2(2R+4M) \right). \]
Setting $\epsilon$ equal to the right-hand side (which really depends only on $R$, since $M$  and $\delta$ are independent of anything), completes the proof.
\end{proof}

A standard argument subdividing an $n$--gon into triangles proves the following.
\begin{corollary} \label{C:cobounded slim n-gons}  Given $R > 0$ let $\epsilon > 0$ be as in Lemma~\ref{L:cobounded slim triangles}.  If $n \geq 3$ and $x_1,\ldots,x_n \in \C^s(\dot S)$ are such that $d_W(x_i,x_j) \leq R$ for all $1 \leq i,j \leq n$, then for all $w \in \sigma(x_i,x_{i+1})$, there exists $j \neq i$ and $w' \in \sigma(x_j,x_{j+1})$ (with all indices taken modulo $n$) such that $d^s(w,w') \leq \lceil \tfrac{n}2 \rceil \epsilon$.
\end{corollary}

For the remainder of the proof (and elsewhere in the paper) it is useful to make the following definition.

\begin{definition} Given $x,y,z \in \C^s(\dot S)$ and $R> 0$, consider the proper witnesses with projection at least $R$:
\[ \Omega_R(x,y) = \{ W \in \Omega_0(\dot S) \mid d_W(x,y) > R\},\]
and set
\[ \Omega_R(x,y,z) = \Omega_R(x,y) \cup \Omega_R(x,z) \cup \Omega_R(y,z).\]
\end{definition}
In words, $\Omega_R(x,y)$ is the set of all proper witness for which $x$ and $y$ have distance greater than $R$.

\begin{lemma} \label{L:at most one large in 3}For any three points $x,y,z \in \C^s(\dot S)$ and $R \geq 2M$, there is at most one $W \in \Omega_R(x,y,z)$ such that
\[ W \in \Omega_{R/2}(x,y) \cap \Omega_{R/2}(x,z) \cap \Omega_{R/2}(y,z).\]
\end{lemma}
\begin{proof} Suppose there exist two distinct
\[ W,W' \in \Omega_{R/2}(x,y) \cap \Omega_{R/2}(x,z) \cap \Omega_{R/2}(y,z).\]
Then by Theorem~\ref{BGIT}, $\partial W, \partial W'$ are (distinct) vertices in any $\C(\dot S)$--geodesic between any two vertices in $\{x,y,z\}$. Choose geodesics $[x,\partial W]_{\dot S}$, $[y,\partial W]_{\dot S}$, and $[z,\partial W]_{\dot S}$, and note that concatenating any two of these (with appropriate orientations) produces a geodesic between a pair of vertices in $\{x,y,z\}$.  Since $\partial W' \neq \partial W$ must also lie on all $\C(\dot S)$--geodesics between these three vertices, it must lie on at least one of the geodesic segments to $\partial W$; without loss of generality, suppose $\partial W' \in [x,\partial W]_{\dot S}$.  If $\partial W'$ is not a vertex of either $[y,\partial W]_{\dot S}$ or $[z,\partial W]_{\dot S}$, then our geodesic from $y$ to $z$ does not contain $\partial W'$, a contradiction.  Without loss of generality, we may assume $\partial W' \in [y,\partial W]_{\dot S}$.  But then the geodesic subsegment between $x$ and $\partial W'$ in $[x,\partial W]_{\dot S}$ together with the geodesic subsegment between $\partial W'$ and $y$ in $[y,\partial W]_{\dot S}$ is also a geodesics (as above) and does not pass through $\partial W$, a contradiction.
\end{proof}

\begin{proof}[Proof of Theorem~\ref{hyp}]  Let $x,y,z \in \C^s(\dot S)$.  By the triangle inequality, if $W \in \Omega_{2M}(x,y)$, then at least one of $d_W(x,z)$ or $d_W(y,z)$ is greater than $M$.  By Lemma~\ref{L:at most one large in 3}, there is at most one $W$ such that {\em both} are greater than $M$.  If such $W$ exists, denote it $W_0$ and write $D_0 = \{W_0\}$; otherwise, write $D_0 = \emptyset$.  Defining
\[ D_x = \{ W \in \Omega_{2M}(x,y,z) \setminus D_0 \mid d_W(x,y) > M, d_W(x,z) > M\} \]
(and defining $D_y$, $D_z$ similarly), we can express $\Omega_{2M}(x,y,z)$ as a disjoint union
\[ \Omega_{2M}(x,y,z) = D_x \sqcup D_y \sqcup D_z \sqcup D_0. \]

By Theorem~\ref{BGIT}, the $\C(\dot S)$--geodesics $[x,y]_{\dot S}$ and $[x,z]_{\dot S}$ contain $\partial W$ for all $W \in D_x$, and we write
\[ D_x = \{W_x^1,W_x^2,\ldots,W_x^{m_x}\} \]
so that $x_1 = \partial W_x^1,x_2 = \partial W_x^2, \ldots, x_{m_x} = \partial W_x^{m_x}$ appear in this order along $[x,y]_{\dot S}$ and $[x,z]_{\dot S}$.  Similarly write
\[ D_y = \{W_y^1,\ldots, W_y^{m_y} \} \mbox{ and } D_z = \{W_z^1,\ldots,W_z^{m_z}\}. \]
The $\C(\dot S)$--geodesic triangle between $x$, $y$, and $z$ must appear as in the examples illustrated in Figure~\ref{F:triangles in C(dot S)}.
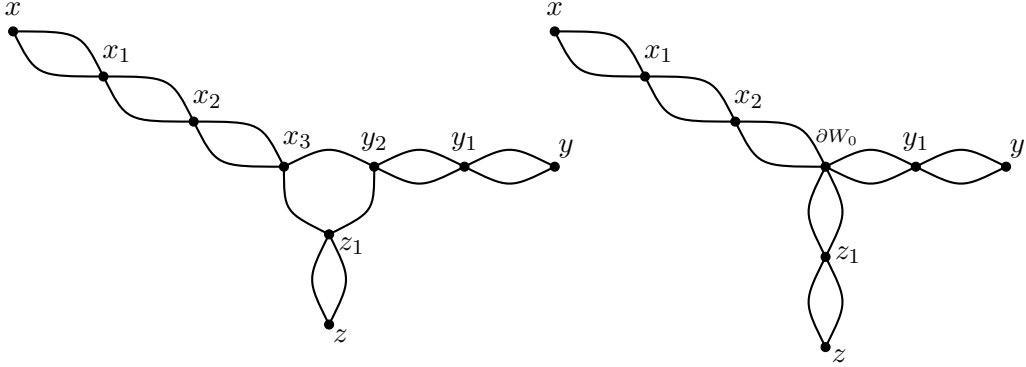
\begin{figure}[h]
\begin{center}
\begin{tikzpicture}[scale = .3]
\draw[thick] (0,0) .. controls (3,0) .. (4,-2);
\draw[thick] (0,0) .. controls (1,-2) .. (4,-2);
\draw[thick] (4,-2) .. controls (7,-2) .. (8,-4);
\draw[thick] (4,-2) .. controls (5,-4) .. (8,-4);
\draw[thick] (8,-4) .. controls (11,-4) .. (12,-6);
\draw[thick] (8,-4) .. controls (9,-6) .. (12,-6);
\draw[thick] (12,-6) .. controls (14,-5) .. (16,-6);
\draw[thick] (16,-6) .. controls (18,-5) .. (20,-6);
\draw[thick] (20,-6) .. controls (22,-5) .. (24,-6);
\draw[thick] (16,-6) .. controls (18,-7) .. (20,-6);
\draw[thick] (20,-6) .. controls (22,-7) .. (24,-6);
\draw[thick] (12,-6) .. controls (12,-8) .. (14,-9);
\draw[thick] (16,-6) .. controls (16,-8) .. (14,-9);
\draw[thick] (14,-9) .. controls (13,-11) .. (14,-13);
\draw[thick] (14,-9) .. controls (15,-11) .. (14,-13);
\draw[fill] (0,0) circle (.2cm);
\draw[fill] (4,-2) circle (.2cm);
\draw[fill] (8,-4) circle (.2cm);
\draw[fill] (12,-6) circle (.2cm);
\draw[fill] (16,-6) circle (.2cm);
\draw[fill] (20,-6) circle (.2cm);
\draw[fill] (24,-6) circle (.2cm);
\draw[fill] (14,-9) circle (.2cm);
\draw[fill] (14,-13) circle (.2cm);
\node at (0,1) {$x$};
\node at (4.6,-1) {$x_1$};
\node at (8.6,-3) {$x_2$};
\node at (12.6,-4.8) {$x_3$};
\node at (16,-4.9) {$y_2$};
\node at (20,-4.9) {$y_1$};
\node at (24.5,-5.2) {$y$};
\node at (15,-9.5) {$z_1$};
\node at (14.5,-13.5) {$z$};
\draw[thick] (18+6,2-2) .. controls (21+6,2-2) .. (22+6,0-2);
\draw[thick] (18+6,2-2) .. controls (19+6,0-2) .. (22+6,0-2);
\draw[thick] (22+6,0-2) .. controls (25+6,0-2) .. (26+6,-2-2);
\draw[thick] (22+6,0-2) .. controls (23+6,-2-2) .. (26+6,-2-2);
\draw[thick] (26+6,-2-2) .. controls (29+6,-2-2) .. (30+6,-4-2);
\draw[thick] (26+6,-2-2) .. controls (27+6,-4-2) .. (30+6,-4-2);
\draw[thick] (30+6,-4-2) .. controls (29+6,-6-2) .. (30+6,-8-2);
\draw[thick] (30+6,-4-2) .. controls (31+6,-6-2) .. (30+6,-8-2);
\draw[thick] (30+6,-8-2) .. controls (29+6,-10-2) .. (30+6,-12-2);
\draw[thick] (30+6,-8-2) .. controls (31+6,-10-2) .. (30+6,-12-2);
\draw[thick] (30+6,-4-2) .. controls (32+6,-3-2) .. (34+6,-4-2);
\draw[thick] (30+6,-4-2) .. controls (32+6,-5-2) .. (34+6,-4-2);
\draw[thick] (34+6,-4-2) .. controls (36+6,-3-2) .. (38+6,-4-2);
\draw[thick] (34+6,-4-2) .. controls (36+6,-5-2) .. (38+6,-4-2);
\draw[fill] (24,0) circle (.2cm); 
\draw[fill] (28,-2) circle (.2cm);
\draw[fill] (32,-4) circle (.2cm);
\draw[fill] (36,-6) circle (.2cm);
\draw[fill] (40,-6) circle (.2cm);
\draw[fill] (44,-6) circle (.2cm);
\draw[fill] (36,-10) circle (.2cm);
\draw[fill] (36,-14) circle (.2cm);
\node at (24,1) {$x$};
\node at (28.6,-1) {$x_1$};
\node at (32.6,-3) {$x_2$};
\node at (36.5,-4.7) {\tiny $\partial W_0$};
\node at (40,-4.9) {$y_1$};
\node at (44.5,-5.2) {$y$};
\node at (37,-9.9) {$z_1$};
\node at (36.6,-14.4) {$z$};
\end{tikzpicture}
\caption{Geodesic triangles in $\C(\dot S)$: Here $x_j = \partial W_x^j$, $y_j = \partial W_y^j$, and $z_j = \partial W_z^j$, and $\{x_1,x_2,x_3\} \subset [x,y]_{\dot S} \cap [x,z]_{\dot S}$, $\{y_1,y_2\} \subset [x,y]_{\dot S} \cap [y,z]_{\dot S}$, and $\{z_1\} \subset [x,z]_{\dot S} \cap [y,z]_{\dot S}$.
The left triangle has $D_0 = \emptyset$, while the triangle on the right has $D_0 = \{W_0\}$, hence $\partial W_0 \in [x,y]_{\dot S} \cap [x,z]_{\dot S} \cap [y,z]_{\dot S}$.  Note: there may be more vertices in common to pairs of geodesics than the vertices $x_j,y_j,z_j$.  Furthermore, there may be various degenerations, e.g.~$D_x= D_0 = \emptyset$, in which case the three bigons in the upper left-hand portion of the left figure disappears and $x_3$ becomes $x$.}
\label{F:triangles in C(dot S)}
\end{center}
\end{figure}

We now subdivide each of the survival paths $\sigma(x,y)$, $\sigma(x,z)$, and $\sigma(y,z)$ into subsegments as follows.
In this subdivision, $\sigma(x,y)$ is a concatenation of witness geodesics for each witness $W$ in $D_x \cup D_y \cup D_0$ and complementary subsegments connecting consecutive such witness geodesics.  The complementary segments are themselves survival paths obtained as concatenations of $\C(\dot S)$--geodesic segments and witness geodesic segments for witnesses for which $d_W(x,y) \leq 2M$.  The paths $\sigma(x,z)$ and $\sigma(y,z)$ are similarly described concatenations.
Applying Lemma~\ref{L:witnessLength}, all of the witness segments that appear in the complementary segments (and are thus {\em not} from witnesses in $\Omega_{2M}(x,y,z)$) have length at most $4M$.

Let $w \in \sigma(x,y)$ be any point.  We must show that there is some $w' \in \sigma(x,z) \cup \sigma(y,z)$ so that $d^s(w,w')$ is uniformly bounded.  There are two cases (which actually divide up further into several sub-cases), depending on whether or not $w$ lies on a witness geodesics for a witness $W \in \Omega_{2M}(x,y,z)$.

Suppose first that $w$ lies on a witness geodesic $[x',y']_W \subset \sigma(x,y)$ for $W \in D_x$.  By definition of $D_x$, $W \in \Omega_M(x,y) \cap \Omega_M(x,z)$, and so there is also a witness geodesic $[x'',z'']_W \subset \sigma(x,z)$.  Since there are $\dot S$--geodesics $[x,x']_{\dot S}, [x,x'']_{\dot S},[y',y]_{\dot S}, [z'',z]_{\dot S}$ so that every vertex has a nonempty projection to $\C(W)$, and since $d_W(y,z) < M$ (again, by definition of $D_x$), Theorem~\ref{BGIT} and the triangle inequality imply
\begin{equation} \label{E:witness endpoints close}
\begin{array}{l} d_W(x',x'') \leq d_W(x',x) + d_W(x,x'') \leq 2M \mbox{ and }\\
d_W(y',z'') \leq d_W(y',y) + d_W(y,z) + d_W(z,z'') \leq 3M.
\end{array}
\end{equation}
So $[x',y']_W$ and $[x'',z'']_W$ are $\C(W)$--geodesics whose starting and ending points are within distance $3M$ of each other.  Since $\C(W)$ is $\delta$--hyperbolic for some $\delta >0$, it follows that there is some $w' \in [x'',z'']_W \subset \sigma(x,z)$ so that $d_W(w,w'') \leq 2 \delta + 3M$.  Since $\C(W)$ is a subgraph of $\C^s(\dot S)$, $d^s(w,w') \leq 2\delta + 3M$.  We can similarly find the required $w'$ if $w$ is in a witness geodesic segment for a witness $W \in D_y$.

Next suppose $w$ lies in the witness geodesic $[x',y']_{W_0} \subset \sigma(x,y)$, for $W_0 \in D_0$ (if $D_0 \neq \emptyset$).  The argument in this sub-case is similar to the previous one, as we now describe.
Let $[x'',z'']_{W_0} \subset \sigma(x,z)$ and $[y'',z']_{W_0}  \subset \sigma(y,z)$ be the $W_0$--geodesic segments.   Arguing as in the proof of (\ref{E:witness endpoints close}), we see that the endpoints of these three geodesic segments in $\C(W)$ satisfy
\[ d_{W_0}(x',x''), d_{W_0}(y',y''), d_{W_0}(z',z'') \leq 2M.\]
Since $\C(W)$ is $\delta$--hyperbolic, we can again easily deduce that for some
\[ w' \in [x'',z'']_{W_0} \cup [y'',z']_{W_0} \subset \sigma(x,z) \cup \sigma(y,z),\]
we have $d^s(w,w') \leq d_{W_0}(w,w') \leq 3\delta + 2M$.

Finally, we assume $w \in \sigma(x,y)$ lies in a complementary subsegment $\sigma(x',y') \subset \sigma(x,y)$ of one of the $\Omega_{2M}(x,y,z)$--witness subsegments of $\sigma(x,y)$ as described above.  Note that $x',y' \in [x,y]_{\dot S} \cap \sigma(x,y)$ both lie in one of the ``bigons'' in Figure~\ref{F:triangles in C(dot S)} (cases (1) and (2) below) or in the single central ``triangle" (case (3) below, which happens when $D_0 = \emptyset$).
Thus, depending on which complementary subsegment we are looking at, we claim that one of the following must hold:
\begin{enumerate}
\item there exists $\sigma(x'',z'') \subset \sigma(x,z)$ so that $d^s(x',x''), d^s(y',z'') \leq 3M$,
\item there exists $\sigma(y'',z'') \subset \sigma(y,z)$ so that $d^s(y',y''), d^s(x',z'') \leq 3M$, or
\item there exists $\sigma(x'',z'') \subset \sigma(x,z)$ and $\sigma(y'',z') \subset \sigma(y,z)$ so that\\ $d^s(x',x''), d^s(y',y''), d^s(z',z'') \leq 3M$.
\end{enumerate}
The proofs of these statements are very similar to the proof in the case that $w \in D_x$ or $D_y$. If $\sigma(x',y')$ is a complementary segment which is part of a bigon and $x'$ is in $\C(W)$ for some $W \in D_x$ (or $x = x''$), then we are in case (1) and we take the corresponding complementary segment $\sigma(x'',z'') \subset \sigma(x,z)$ of the bigon with $x'' \in \C(W)$ (or $x'' = x$).  It follows that all vertices of $[x',y]_{\dot S}$, $[y,z]_{\dot S}$, and $[z,x'']_{\dot S}$ have non-empty projections to $W$, so by Theorem~\ref{BGIT} and the triangle inequality we have
\[ d^s(x',x'') \leq d_W(x',x'') \leq d_W(x',y) + d_W(y,z) + d_W(z,x'') \leq 3M. \]
On the other hand, $y',z'' \in \C(W')$ for some $W' \in D_x \cup D_0$  and similarly
\[ d^s(y',z'') \leq d_{W'}(y',z'') \leq d_{W'}(y',x) + d_{W'}(x,z'') \leq 2M < 3M,\]
and so the conclusion of (1) holds.  If $y' \in \C(W)$ for some $W \in D_y$, then a symmetric argument proves (2) holds.  The only other possibility is that $D_0 = \emptyset$, $x' \in \C(W)$, and $y' \in \C(W')$, where $W \in D_x$ and $W' \in D_y$, so that $\sigma(x',y')$ is a segment of the ``triangle".  A completely analogous argument proves that condition (3) holds.

In any case, note that the two subsegments of the bigon (respectively, three segments of the central triangle), together with segments in curve complexes of proper witnesses give a quadrilateral (respectively, hexagon) of survival paths.  Furthermore, by the triangle inequality and application of Theorem~\ref{BGIT}, we see that there is a uniform bound $R >0$ to the projections to all proper witnesses of the vertices of this quadrilateral (respectively, hexagon).  Let $\epsilon > 0$ be the constant from Lemma~\ref{L:cobounded slim triangles} for this $R$.  By Corollary~\ref{C:cobounded slim n-gons}, there is some $w'$ on one of the other sides of this quadrilateral/hexagon so that $d^s(w,w') \leq 3\epsilon$.  It may be that $w'$ is in $\sigma(x,z)$ or $\sigma(y,z)$, or that it lies in one of the witness segments.  As described above, these segments have length at most $3M$, and so in this latter case, we can find $w'' \in \sigma(x,z) \cup \sigma(y,z)$ with $d^s(w,w'') \leq 3\epsilon + 3M$.

Combining all the above, we see that there is always some $w' \in \sigma(x,z) \cup \sigma(y,z)$ with $d^s(w,w')$ bounded above by
\[ \max\{3 \epsilon + 3M, 2 \delta + 3M, 4 \delta + 2M\}.\]
This provides the required uniform bound on thinness of survival paths, and completes the proof of the theorem.
\end{proof}

\section{Distance Formula} \label{S:distance formula}

In this section we prove the following theorem.
\begin{theorem} \label{dist}
 For any $k \geq \max\{M,24\}$, there exists $K \geq 1$, $C \geq 0$ so that
\[ d^s(x,y) \stackrel{K,C}{\asymp}  \sum_{W \in \Omega(\dot S)} \lcut d_{W}(x, y) \rcut_k,\]
for all $x,y \in \sC$.
\end{theorem}
Recall that here $x \stackrel{K,C}{\asymp} y$ is shorthand for the condition $\frac{1}{K}(x - C) \leq y \leq Kx + C$ and that $\lcut x \rcut_k = x$ if $x \geq k$ and $0$, otherwise.  Note that we have already proved an upper bound on $d^s(x,y)$  of the required form in Corollary~\ref{L:upper bound formula} and thus we need only prove the lower bound.

\begin{remark} As with Theorem~\ref{hyp}, another approach to this theorem would be to follow Masur-Schleimer \cite{MSch1} or Vokes \cite{Vokes}.  As with Theorem~\ref{hyp} we give a proof using survival paths, which is straightforward and elementary.
\end{remark}

One of the main ingredients in our proof is the following due to Behrstock  \cite{BehrTH} (see \cite{Joh1} for the version here).
\begin{lemma}[Behrstock Inequality]\label{Behr}
Assume that $W$ and $W'$ are witnesses for $\sC$ and $u\in \C(\dot S)$ with nonempty projection to both $W$ and $W'$. Then
\[ d_{W}(u, \partial W') \geq 10 \Rightarrow   d_{W'}(u, \partial W) \leq 4 \]
\end{lemma}

We will also need the following application which we use to provide an ordering on the witnesses for a pair $x,y \in \sC$ having large enough projection distances.  A more general version was proved in \cite{BBML} (see also \cite{CLM}) and is related to the partial order on domains of hierarchies from \cite{MM2}.  The version we will use is the following.
\begin{proposition} \label{order} Suppose $k \geq 14$ and $W,W'$ are witnesses in the set $\Omega_k(x,y)$. Then the following are equivalent:
\begin{center}
\begin{tabular}{lll}
(1) $d_{W'}(y, \partial W)\geq10$ \quad \quad  & (2) $d_{W}(y, \partial W')\leq4$\\
(3) $d_{W}(x, \partial W')\geq10$  & (4) $d_{W'}(x, \partial W)\leq4$
\end{tabular}
\end{center}
\end{proposition}
\begin{proof} By Lemma  \ref{Behr} we have $(1) \Rightarrow (2)$  and $(3) \Rightarrow (4)$. To prove  $(2) \Rightarrow (3)$
we use triangle inequality:
\[ d_{W}(x, \partial W')\geq d_{W}(x, y)-d_{W}(y, \partial W') \geq k - 4 \geq 10\]
since $k \geq 14$.  The proof of $(4) \Rightarrow (1)$ is identical to the proof that $(2) \Rightarrow (3)$.
\end{proof}

\begin{definition} For any $k\geq 14$, we define a relation $<$ on $\Omega_k(x,y)$, declaring $W < W'$ for $W, W' \in \Omega_k(x,y)$, if
any of the equivalent statements of the Proposition \ref{order} is satisfied.
\end{definition}

\begin{lemma}  For any $k \geq 20$, the relation $<$ is a total order on $\Omega_k(x,y)$.
\end{lemma}
\begin{proof}
We first prove that any two element $W , W' \in \Omega_k(x,y)$ are ordered. If not, then that means Proposition~\ref{order} (3) fails to hold as stated, or with $y$ replacing $x$, and thus we have  $d_{W}(y, \partial W') < 10 $ and $d_{W}(x, \partial W') < 10$. Hence,
\[d_{W}(x,y) \leq d_{W}(x, \partial W')+d_{W}(y, \partial W') < 20 \leq k\]
which contradicts the assumption that $W \in \Omega_k(x,y)$.

The relation is clearly anti-symmetric, so it remains to prove that it is transitive.  To that end, let $W < W' < W''$ in $\Omega_k(x,y)$, and we assume $W \not < W''$, hence $W'' < W$.  Since $W< W'$ and $W'' < W$, we have $d_W(y,\partial W') \leq 4$ and $d_W(x,\partial W'') \leq 4$.  So by the triangle inequality
\[ d_W(\partial W',\partial W'' ) \geq d_W(x,y) - d_W(y,\partial W') - d_W(x,\partial W'') \geq k - 8 > 10.\]
Then by Lemma~\ref{Behr}, we have
\[ d_{W'}(\partial W,\partial W'') \leq 4. \]
So, appealing to the fact that $W< W'$ and $W' < W''$ and Proposition~\ref{order} the triangle inequality implies
\[ 20 \leq k \leq d_{W'}(x,y) \leq d_{W'}(x,\partial W) + d_{W'}(\partial W,\partial W'') + d_{W'}(\partial W'',y) \leq 12, \]
a contradiction.
\end{proof}

The next lemma is also useful in the proof of Theorem \ref{dist}.
\begin{lemma}\label{aux} Let $x,y,u \in \sC$, $W,W' \in \Omega_k(x,y)$ with $k \geq 20$, and  $W<W'$.  Then,
\[ d_{W}(u, y) \geq 14 \Rightarrow d_{W'}(u, x) \leq 8\]
\end{lemma}
\begin{proof} From our assumptions, the definition of the order on $\Omega_k(x,y)$, and the triangle inequality we have
\[ d_W(u,\partial W') \geq d_W(u,y) - d_W(y,\partial W') \geq 14 - 4 = 10.\]
By Lemma~\ref{Behr}, we have $d_{W'}(u,\partial W) \leq 4$.  Thus, by the definition of the order on $\Omega_k(x,y)$ and the triangle inequality, we have
\[ d_{W'}(u,x) \leq d_{W'}(u,\partial W) + d_{W'}(\partial W,x) \leq 4 + 4 = 8.\]
\end{proof}

We are now ready to prove the lower bound in Theorem \ref{dist}, which we record in the following proposition.

\begin{proposition}\label{lowerbnd} Fix $k \geq 24$.
Given $x,y \in \sC$ we have
\[d^s(x, y) \geq \frac1{96} \sum_{W \in \Omega(\dot S)} \lcut  d_W(x,y) \rcut_k    \]
\end{proposition}
\begin{proof} Let  $[x,y]$ be a geodesic between $x,y$ in $\sC$, and denote its vertices
\[ x = x_0,x_1,\ldots,x_{n-1},x_n = y.\]
So, $n = d^s(x,y)$ is the length of $[x,y]$.
Let $m = |\Omega_k(x,y)|$, suppose $m > 0$, and write
\[\Omega_k(x,y)= \{W_1<W_2<\cdots <W_m\}.\]

For each $1 \leq j < m$ let $0 \leq i_j \leq n$ be such that $d_{W_j}(x_{i_j}, y) \geq 14$ and $d_{W_j}(x_{\ell}, y) \leq 13$ for all $\ell > i_j$.  That is, $x_{i_j}$ is the last vertex $z \in [x,y]$ for which $d_{W_j}(z,y) \geq 14$.  Then, if $j'>j$, so $W_j < W_{j'}$, Lemma~\ref{aux} implies $d_{W_{j'}}(x_{i_j}, x) \leq 8$ and so
\[ d_{W_{j'}}(x_{i_j},y) \geq d_{W_{j'}}(x,y) - d_{W_{j'}}(x_{i_j},x) \geq k - 8 \geq 24-8 = 16.\]
Since the projection $\pi_{W_{j'}}$ is $2$--Lipschitz (see Proposition~\ref{P:2 Lipschitz}) and $x_{i_j}$ and $x_{i_j+1}$ are distance $1$ in $\sC$, we have
\[ d_{W_{j'}}(x_{i_j+1},y) \geq d_{W_{j'}}(x_{i_j},y) - d_{W_{j'}}(x_{i_j},x_{i_j+1}) \geq 16-2 \geq 14.\]
Therefore, $i_1<i_2<\cdots <i_{m-1}$. Set $i_0 = 0$ and $i_m = n$.

Given $1 \leq j < m$, $d_{W_j}(x_{i_j+1},y) \leq 13$ and again appealing to Proposition~\ref{P:2 Lipschitz}, we have
\[ d_{W_j}(x_{i_j},y) \leq d_{W_j}(x_{i_j},x_{i_j+1}) + d_{W_j}(x_{i_j+1},y) \leq 2+ 13 \leq 15.\]
Observe this inequality is trivially true for $j = m$ since $y = x_n = x_{i_m}$ and so the left hand side is at most $2$ in this case.
Another application of Lemma~\ref{aux} implies $d_{W_j}(x,x_{i_{j-1}}) \leq 8$ for all $1 \leq j \leq m$ (the case $j = 1$ is similarly trivial).  Therefore
\begin{equation} \label{E:for the lower bound} d_{W_j}(x_{i_{j-1}},x_{i_j}) \geq d_{W_j}(x,y) - d_{W_j}(x,x_{i_{j-1}}) - d_{W_j}(x_{i_j},y) \geq d_{W_j}(x,y) - 23,
\end{equation}
for all $1 \leq j \leq m$.

Appealing one more time to Proposition~\ref{P:2 Lipschitz}, together with Inequality~(\ref{E:for the lower bound}), we have
\[ d^s(x, y) = n = \sum_{j=1}^m i_j-i_{j-1} \geq \frac12 \sum_{j=1}^m d_{W_j}(x_{i_{j-1}},x_{i_j})  \geq \frac{1}{2}\sum_{j=1}^{m}(d_{W_j}(x,y)-23) \]

Next, observe that since $d_{W_j}(x,y)\geq k \geq 24$ we have
\[d_{W_j}(x,y)-23 \geq   \frac{1}{24}d_{W_j}(x,y).\]
Since $\C^s(\dot S) \subset \C(\dot S)$ is a subcomplex, we have $d^s(x,y) \geq d_{\dot S}(x,y)$ and so
\[ 2 d^s(x,y) \geq d_{\dot S}(x,y)  +  \frac1{48} \sum_{W \in \Omega_k(x,y)} d_W(x,y)  \geq \frac1{48} \sum_{W \in \Omega(\dot S)} \lcut d_W(x,y) \rcut_k.\]
\end{proof}

\begin{proof}[Proof of Theorem \ref{dist}]
Given $k\geq \max\{M,24\}$, let $K = \max\{2k,96\}$ and $C = 2k^2$.  The theorem then follows from Corollary~\ref{L:upper bound formula}  and Proposition~\ref{lowerbnd}.
\end{proof}
As a consequence of the Theorem \ref{dist} we have the following two facts.

\begin{corollary} \label{C:witnesses qi embed} Given a witness  $W\subset S$, the inclusion map $\C(W)\hookrightarrow \sC$ is a quasi-isometric embedding.
\end{corollary}

\begin{corollary}\label{qd}Survival paths are uniform quasi-geodesics in $\sC$.

\end{corollary}
Moreover, we have
\begin{lemma} \label{L:reparameterize survival} Survival paths can be reparametrized to be uniform quasi-geodesics in $\calC(\dot S)$.
\end{lemma}
\begin{proof}  Let $\sigma(x,y)$ be a survival path with main geodesic $[x,y]_{\dot S}$.  For every proper witness $W \subsetneq \dot S$, if there is a $W$--witness geodesic segment in $\sigma(x,y)$, we reparametrize along this segment so that it is traversed along an interval of length $2$.  Since such $W$--witness geodesic segments replaced geodesic subsegments of $[x,y]_{\dot S}$ of length $2$, and since they lie in the $1$--neighborhood of $\partial W$, this clearly defines the required reparametrization.
\end{proof}

\begin{corollary} \label{C:infinite survival qgeod} Any infinite survival path is a uniform quasi-geodesic.
\end{corollary}
\begin{proof} This is immediate from Corollary~\ref{qd} and Proposition~\ref{P:exhaustion by survival}.
\end{proof}

\section{Boundary of the surviving curve complex $\sC$} \label{S:boundary}

Recall that we denote the disjoint union of ending lamination spaces of all witnesses by
\[\calE\calL^s(\dot S):=\bigsqcup_{W \in \Omega(\dot S)} \calE \calL(W).\]
We call this the \textit{space of surviving ending laminations} of $\dot S$, and give it the coarse Hausdorff topology.

In this section we will prove Theorem~\ref{survivalendinglam} from the introduction.  In fact, we will prove the following more precise version, that will be useful for our purposes.


\begin{theorem} \label{T:boundary ending precise} There exists a homeomorphism $\calF \colon \partial \C^s(\dot S) \to \EL^s(\dot S)$ such that for any sequence $\{\alpha_n\} \subset \C^s(\dot S)$, $\alpha_n \to x$ in $\bar \C^s(\dot S)$ if and only if $\alpha_n \xrightarrow{\text{CH}} \calF(x)$.
\end{theorem}

We denote the Gromov product of $\alpha,\beta \in \sC$ based at $o \in \sC$ by $\la \alpha,\beta \ra_o^s$, and recall that the Gromov boundary $\partial \sC$ of $\sC$ is defined to be the set of equivalence classes of sequences $\{\alpha_n\}$ which converge at infinity with respect to $\la \, \, , \, \, \ra_o^s$.  Throughout the rest of this section we will use (without explicit mention) the fact that the Gromov product between a pair of point (in any hyperbolic space) is uniformly estimated by the minimal distance from the basepoint to a point on a uniform quasi-geodesic between the points.

For each proper witness $W \subsetneq \dot S$, Corollary~\ref{C:witnesses qi embed} implies that  $\partial \C(W)$ embeds into $\partial \sC$.  Likewise, Corollary~\ref{C:infinite survival qgeod}, combined together with the fact that $d_{\dot S} \leq d^s$ implies that $\partial \C(\dot S)$ also embeds $\partial \sC$.  Using these embeddings, we view $\partial \C(W)$ as a subspace of $\partial \sC$, for all witnesses $W \subseteq \dot S$.  The next proposition says that the subspaces are all disjoint.
\begin{proposition} \label{P:Witness boundaries embed disjoint} For any two witnesses $W \neq W'$ for $\dot S$, $\partial \C(W) \cap \partial \C(W')= \emptyset$.
\end{proposition}
\begin{proof} Let $x\in \partial \C(W)$ and $x'\in \partial \C(W')$. Then by Lemma~\ref{L:visual on witness boundaries}, there is a bi-infinite survival path $\sigma(x,x')$ and by Corollary~\ref{C:infinite survival qgeod} this survival path is a quasi geodesic. Hence $x\neq x'$.
\end{proof}

We now have a natural inclusion of the disjoint union of Gromov boundaries
\[ \bigsqcup_{W \in \Omega(\dot S)} \partial \C(W) \subset \partial \sC.\]
In fact, this disjoint union accounts for the entire Gromov boundary.
\begin{lemma} \label{L:bijection} We have
\[ \bigsqcup_{W \in \Omega(\dot S)} \partial \C(W) = \partial \sC.\]
\end{lemma}
\begin{proof} Let $x\in  \partial \sC$ and $\alpha_n\rightarrow x\in \partial\sC$, and we assume without loss of generality that $\{\alpha_n\}$  is a quasi-geodesic in $\C^s(\dot S)$ and that the first vertex is the basepoint $\alpha_0 = o$.  If $d_{\dot S}(\alpha_n, o)\rightarrow \infty $ as $n \rightarrow \infty$, then given $R >0$, let $N > 0$ be such that $d_{\dot S}(\alpha_n,o) \geq R$ for all $n \geq N$.  For any $m \geq n \geq N$, the subsegment of the quasi-geodesics, $\{\alpha_n,\alpha_{n+1},\ldots,\alpha_m\}$, is some uniformly bounded distance $D$ to $\sigma(\alpha_n,\alpha_m)$ in $\C^s(\dot S)$, by hyperbolicity and Corollary~\ref{qd}.  Therefore, the distance in $\C(\dot S)$ from any point of $\sigma(\alpha_n,\alpha_m)$ to $o$ is at least $R - D$.  So the distance from any point of $[\alpha_n,\alpha_m]_{\dot S}$ to $o$ is at least $R - D - 1$.  Letting $R \to \infty$, it follows that  $\la \alpha_n, \alpha_m\ra_o \rightarrow \infty$ in $\Cdot$.  Consequently,  $\{\alpha_n\}$ converges to a point in $\partial \C(\dot S)$, so $x \in \partial \C(\dot S)$.   For the rest of the proof, we may assume that $d_{\dot S}(\alpha_n,o)$ is  bounded by some constant $0 < R < \infty$ for all $n$. 


By the distance formula \ref{dist},
\[ d^s(\alpha_0,\alpha_n) \stackrel{K,C}{\asymp}  \sum_{W \in \Omega(\dot S)} \lcut d_{W}(\alpha_0, \alpha_n) \rcut_k,\]
and since $d^s(\alpha_0, \alpha_n) \rightarrow \infty$, we can find a witness $W_n$ for $\sigma(\alpha_0,\alpha_n)$ for each $n \in \mathbb N$, so that we have $d_{W_n}(\alpha_0, \alpha_n) \rightarrow \infty$ as $n \to \infty$.

We would like to show that there is a unique witness $W$ such that  $d_W(\alpha_0, \alpha_n)\rightarrow \infty$.
To do that, let $h > 0$ be  the maximal Hausdorff distance in $\C^s(\dot S)$ between $\sigma(\alpha_0,\alpha_n)$ and $\{\alpha_k\}_{k=0}^n$, for all $n \geq 0$ (which is finite by hyperbolicity of $\C^s(\dot S)$ and Corollary~\ref{qd}).
\begin{claim} \label{Claim:large persists} Given $n \in \mathbb N$, if $d_W(\alpha_0,\alpha_n) \geq  M + 1 + 2(h+1)$ for some witness $W \subsetneq \dot S$, then $W$ is a witness for $\sigma(\alpha_0,\alpha_m)$, for all $m \geq n$.
\end{claim}
\begin{proof} Let $\alpha_n' \in \sigma(\alpha_0,\alpha_m)$ be such that $d^s(\alpha_n,\alpha_n') \leq h$.  If $W$ is not a witness for $\sigma(\alpha_0,\alpha_m)$, then every vertex of the main geodesic $[\alpha_0,\alpha_m]_{\dot S}$ of $\sigma(\alpha_0,\alpha_m)$ has nonempty project to $W$.  Furthermore, the geodesic in $\C^s(\dot S)$ from $\alpha_n$ to $\alpha_n'$ of length at most $h$ can be extended to a path in $\C(\dot S$) to $[\alpha_0,\alpha_m]_{\dot S}$ of length at most $h+1$ such that every vertex has nonempty project to $W$.  By Proposition~\ref{P:2 Lipschitz} and the triangle inequality we have,
\[ |d_W(\alpha_0,\alpha_n) - d_W(\alpha_0,\alpha_n')| \leq 2(h+1).\]
If $\alpha_n' \in [\alpha_0,\alpha_m]_{\dot S}$, then since every vertex of this geodesic has nonempty projection to $W$, it follows that $d_W(\alpha_0,\alpha_n') < M$.  If $\alpha_n' \not \in [\alpha_0,\alpha_m]_{\dot S}$, then there is a witness  $W'$ for $[\alpha_0, \alpha_m]_{\dot S}$ such that $d_{\dot S}(\alpha'_n, \partial W')=1$, and as a result $d_W(\alpha_0,\alpha_n') < M+1$. In any case,
\[ d_W(\alpha_0,\alpha_n) \leq d_W(\alpha_0,\alpha_n') + 2(h+1) < M + 1+ 2(h+1),\]
which is a contradiction.  This proves the claim.
\end{proof}

Since $d_{W_n}(\alpha_0,\alpha_n) \to \infty$, there exists $n_0 >0$ such that $d_{W_{n_0}}(\alpha_0,\alpha_{n_0}) \geq 2(h+1) + M + 1$, and hence for all $m \geq n_0$, $W_{n_0}$ is a witness for $\sigma(\alpha_0,\alpha_m)$.
Let $\Omega([\alpha_0,\alpha_n]_{\dot S})$ be the set of proper witnesses for $[\alpha_0,\alpha_n]_{\dot S}$, and set
\[ \Omega_n = \bigcap_{m =n}^\infty \Omega([\alpha_0,\alpha_m]_{\dot S}).\]
Note that $\Omega_n \subset \Omega_{n+1}$ for all $n$ and that $\Omega_n$ is nonempty for all $n \geq n_0$.  Since each $\Omega_n$ contains no more than $R/2$ elements by Corollary~\ref{qd}, the (nested) union $\Omega_\infty$ is given by $\Omega_\infty = \Omega_N$  for some $N \geq n_0$.
The boundaries of the witnesses in $\Omega_\infty$ lie on the geodesic $[\alpha_0,\alpha_m]_{\dot S}$ for all $m \geq N$, and we let $W_\infty \in \Omega_\infty$ be the one furthest from $\alpha_0$.  Without loss of generality, we may assume that $[\alpha_0,\alpha_m]_{\dot S}$ and $[\alpha_0,\alpha_{m'}]_{\dot S}$ all agree on $[\alpha_0,\partial W_\infty]$, for all $m,m' \geq N$.

For any $m \geq N$ and any witness $W$ of $[\alpha_0,\alpha_N]_{\dot S}$ with $\partial W$ {\em further} from $\partial W_\infty$, note that $d _W(\alpha_0,\alpha_m) < M+1+2(h+1)$: otherwise, by Claim~\ref{Claim:large persists} $W$ would be a witness for $[\alpha_0,\alpha_{m'}]_{\dot S}$ for all $m' \geq m$ and so $W \in \Omega_\infty$ with $\partial W$ further from $\alpha_0$ than $\partial W_\infty$, a contradiction to our choice of $W_\infty$.

For any $n \geq N$, let $\beta_n$ be the last vertex of $\sigma(\alpha_0,\alpha_n)$ in $\C(W_\infty)$.  By the previous paragraph together with Theorem~\ref{BGIT} and the bound
\[ d_{\dot S}(\beta_n,\alpha_n) \leq d_{\dot S}(\alpha_0,\alpha_n) \leq R/2,\]
we see that the subpath of $\sigma(\alpha_0,\alpha_n)$ from $\beta_n$ to $\alpha_n$ has length bounded above by some constant $C >0$, independent of $n$.  In particular, $d^s(\alpha_n,\beta_n) \leq C$. Therefore, $\alpha_n$ and $\beta_n$ converge to the same point $x$ on the Gromov boundary of $\C^s(\dot S)$.  Since $\beta_n \in \C(W_\infty)$, which is quasi-isometrically embedded in $\C^s(\dot S)$, it follows that $x \in \partial \C(W_\infty)$, as required.
\end{proof}

The next Lemma provides a convenient tool for deciding when a sequence in $\C^s(\dot S)$ converges to a point in $\partial \C(W)$, for some proper witness $W$.
\begin{lemma}\label{L:proj and conv2} Given $\{\alpha_n\} \subset \sC$ and $x\in \partial \calC(W)$ for some witness $W$, then $\alpha_n {\rightarrow} x$ if and only if $\pi_W(\alpha_n) \rightarrow x$.
\end{lemma}
\begin{proof}  Throughout, we assume $o=\alpha_0$, the basepoint, which without loss of generality we assume lies in $W$, and let $\{\beta_n\} \subset \C(W)$ be any sequence converging to $x$, so that for the Gromov product in $\C(W)$ we have $\langle \beta_n,\beta_m \rangle_o^W \to \infty$ as $n,m \to \infty$.

Since $\sigma(\alpha_n,\beta_n)$ is a uniform quasi-geodesic by Corollary~\ref{qd} it follows that
\[ d^s(o,\sigma(\alpha_n,\beta_n))  \asymp  \langle \alpha_n,\beta_m \rangle^s_o ,\]
with uniform constants (where the distance on the left is the minimal distance from $o$ to the survival path).

Let $\delta_{n,m}$ be the first point of intersection of $\sigma(\alpha_n,\beta_m)$ with $\C(W)$ (starting from $\alpha_n$).  By Lemma~\ref{L:witnessLength}, $d_W(\delta_{n,m},\alpha_n) < M$.  Consequently, because $\delta_{n,m} \in \C(W)$ and $\pi_W(\alpha_n) \subset \C(W)$, this means
\[ d^s(\delta_{n,m},\pi_W(\alpha_n)) \leq d_W(\delta_{n,m},\pi_W(\alpha_n) ) = d_W(\delta_{n,m},\alpha_n)   < M.\]
Therefore, by hyperbolicity, the $\C^s(\dot S$)--geodesic from (any curve in) $\pi_W(\alpha_n)$ to $\beta_n$ lies in a uniformly bounded neighborhood of $\sigma(\alpha_n,\beta_m)$, and so
\[ \langle \pi_W(\alpha_n),\beta_m \rangle^s_o \asymp d^s(o,[\pi_W(\alpha_n),\beta_m]) \succeq d^s(o,\sigma(\alpha_n,\beta_m)) \asymp \langle \alpha_n,\beta_m \rangle^s_o. \]
If $\alpha_n \to x$, then the right-hand side of the above coarse inequality tends to infinty, and hence so does the left-hand side.  This implies $\pi_W(\alpha_n) \to x$.

Next suppose that $\pi_W(\alpha_n) \to x \in \partial \C(W)$.  As above, we have
\[ \langle \alpha_n,\pi_W(\alpha_n) \rangle_o^s \asymp d^s(o,[\alpha_n,\pi_W(\alpha_n)]),\]
and so it suffices to show that the right-hand side tends to infinity as $n \to \infty$.  Since $\pi_W(\alpha_n) \to x \in \partial \C(W)$, we have $d_W(o,\alpha_n) \to \infty$, and setting $\delta_n$ to be the first point of $\sigma(\alpha_n,o)$ in $\C(W)$, Lemma~\ref{L:witnessLength} implies that $d^s(\delta_n,\pi_W(\alpha_n)) < M$.  Therefore, $\sigma(\alpha_n,o)$ passes within $d^s$--distance $M$ of $\pi_W(\alpha_n)$ on its way to $o$.  Since $\sigma(\alpha_n,o)$ is a uniform quasi-geodesic by Corollary~\ref{qd}, it follows that $[\alpha_n,\pi_W(\alpha_n)]$ is uniformly Hausdorff close to the initial segment $J_n \subset \sigma(\alpha_n,o)$ from $\alpha_n$ to $\delta_n$.  Since the closest point of $J_n$ to $o$ is, coarsely, the point $\delta_n$, which is uniformly close to $\pi_W(\alpha_n)$, we have
\begin{eqnarray*} \langle \alpha_n,\pi_W(\alpha_n) \rangle_o^s & \asymp & d^s(o,[\alpha_n,\pi_W(\alpha_n)])  \\
 & \asymp & d^s(o,J_n)  \, \, \asymp \, \, d^s(o,\delta_n) \, \, \asymp \, \, d^s(o,\pi_W(\alpha_n)) \, \, \asymp \, \, d_W(o,\alpha_n) \to \infty.
\end{eqnarray*}
Therefore, $\alpha_n$ and $\pi_W(\alpha_n)$ converge together to $x \in \partial \C(W)$.  This completes the proof.
\end{proof}

\begin{proof}[Proof of Theorem~\ref{T:boundary ending precise}] By Lemma~\ref{L:bijection}, for any $x \in \partial \sC$, there exists a witness $W \subseteq \dot S$ so $x \in \partial \C(W)$.  Let $\mathcal F(x) = \mathcal F_W(x)$, where $\mathcal F_W \colon \partial \C(W) \to \EL(W)$ is the homeomorphism given by Theorem~\ref{T:Klarreich}.  This defines a bijection $\mathcal F \colon \partial \sC \to \EL^s(\dot S)$.

We let $x \in \partial \sC$ with $\alpha_n \to x$ in $\bar \C^s(\dot S)$, and prove that $\alpha_n$ coarse Hausdorff converges to $\mathcal F(x)$.  Let $W \subseteq \dot S$ be the witness with $x \in \partial \C(W)$.  According to Proposition~\ref{L:proj and conv2}, $\pi_W(\alpha_n) \to x$ in $\bar \C(W)$.  By Theorem~\ref{T:Klarreich}, $\pi_W(\alpha_n) \stackrel{CH}{\to} \mathcal F_W(x) = \mathcal F(x)$, and by Lemma~\ref{L:proj and conv}, $\alpha_n \stackrel{CH}{\to} \mathcal F(x)$, as required.

To prove the other implication, we suppose that $\alpha_n \stackrel{CH}{\to} \mathcal L$, for some $\mathcal L \in \EL^s(\dot S)$, and prove that $\alpha_n \to x$ in $\bar \C^s(\dot S)$ where $\calF(x) = \calL$.  Let $W \subseteq S$ be the witness with $\mathcal L \in \EL(W)$.  By Lemma~\ref{L:proj and conv}, $\pi_W(\alpha_n) \stackrel{CH}{\to} \mathcal L$.  By Theorem~\ref{T:Klarreich}, $\pi_W(\alpha_n) \to x$ in $\bar \C(W)$ where $\mathcal F_W(x) = \mathcal L$. By Proposition~\ref{L:proj and conv2}, $\alpha_n \to x$ in $\bar \C^s(\dot S)$ and $\mathcal F(x) = \mathcal F_W(x) = \mathcal L$.

All that remains is to show that $\mathcal F$ is a homeomorphism.  Throughout the remainder of this proof, we will frequently pass to subsequences, and will reindex without mention.
We start by proving that $\mathcal F$ is continuous. Let $\{x^n\} \subset \partial \sC$ with $x^n \to x$ as $n \to \infty$.  Pass to {\em any} Hausdorff convergence subsequence so that $\calF(x^n) \xrightarrow{\text{H}}\calL$ for some lamination $\calL$.  If we can show that $\calF(x) \subseteq \calL$, then this will show that the original sequence coarse Hausdorff converges to $\mathcal F(x)$, and thus $\calF$ will be continuous.

For each $n$, let $\{ \alpha_k^n \}_{k=1}^\infty \subset \C^s(\dot S)$ be a sequence with $\alpha_k^n \rightarrow x^n$ as $k\rightarrow \infty$.  Since $x^n \to x$, we may pass to subsequences so that for any sequence $\{k_n\}$, we have $\alpha_{k_n}^n \to x$ as $n \to \infty$.
From the first part of the argument, $\alpha_k^n \xrightarrow{\text{CH}}\calF(x^n)$ as $k \to \infty$, for all $n$.  For each $n$, pass to a subsequence so that $ \alpha_k^n \xrightarrow{\text{H}}\calL_n$, thus $\calF(x^n)\subseteq \calL_n$.  By passing to yet a further subsequence for each $n$, we may assume $d_H(\alpha_k^n, \calL_n) < \frac{1}n$ for all $k$; in particular, this holds for $k=1$.
Now pass to a subsequence of $\{\calL_n\}$ so that $\calL_n \stackrel{H}{\to} \calL_o$ for some lamination $\calL_o$, it follows that $\alpha_1^n \stackrel{H}\to \calL_o$, as $n \to \infty$.  Since $\alpha_1^n \to x$ (from the above, setting $k_n=1$ for all $n$), this implies that $\calF(x) \subseteq \calL_o$.


Since $\calF(x^n)\subseteq \calL_n$ we have $\calL\subseteq \calL_o$.
If $\calF(x) \in \EL(\dot S)$ then it is the unique minimal sublamination of $\calL_o$, and since $\calL \subseteq \calL_o$, we have $\calF(x) \subseteq \calL$.  If $\calF(x) \in \EL(W)$ for some proper witness, then either $\calF(x)$ is the unique minimal sublamination of $\calL_o$, or $\calL_o$ contains $\calF(x) \cup \partial W$.  Since $\partial W$ does not intersect the interior of $W$, whereas $\calL \subset \calL_o$ is a sublamination that does nontrivially intersects the interior of $W$, it follows that $\calF(x) \subseteq \calL$. Therefore we have $\calF(x) \subseteq \calL$ in both cases, and so $\calF$ is continuous.

To prove continuity of $\calG=\calF^{-1}$, suppose $\calL_n \xrightarrow{\text{CH}} \calL$, and we must show $\calG(\calL_n) \to \calG(\calL)$. We first pick a sequence of curves $\alpha^n_k$ such that $\alpha^n_k \rightarrow \calG(\calL_n)$ in $\bar \C^s(\dot S)$. Then, $\alpha^n_k \xrightarrow{\text{CH}}\calL_n$ as $k\rightarrow \infty$, by the first part of the proof, and after passing to subsequences as necessary, we may assume: (i) $\alpha_k^n \stackrel{H}\to \calL_n'$ as $k \to \infty$, and hence $\calL_n \subseteq \calL_n'$ for all $n$; (ii) $d_H(\alpha_k^n,\calL_n') < \tfrac1n$ for all $k$; and (iii) $\langle \alpha_k^n,\alpha_\ell^n \rangle_o \geq \min\{k,\ell\} + n$, for all $k,\ell,n$.

Now pass to {\em any} Hausdorff convergent subsequence $\calL_n' \stackrel{H}\to \calL'$. It suffices to show that for this subsequence $\calG(\calL_n) \to \calG(\calL)$.  Observe that we also have
$\calL \subseteq \calL'$ and by (ii) above we also have $\alpha_{k_n}^n \to \calL'$ as $n \to \infty$, for {\em any} sequence $\{k_n\}$.  Thus, for example, we can conclude that $\alpha_1^n \stackrel{CH}\to \calL$, and so by the first part of the proof we have $\alpha_1^n \to \calG(\calL)$.

As equivalence classes of sequences, we thus have $\calG(\calL_n) = [\{\alpha_k^n\}]$ and $\calG(\calL) = [\{\alpha_1^m\}]$.  We further observe that by hyperbolicity and the conditions above, for all $k,n,m$ we have
\[ \langle \alpha_1^m,\alpha_k^n \rangle_o \succeq \min\{ \langle \alpha_1^m,\alpha_1^n \rangle_o,\langle \alpha_1^n,\alpha_k^n \rangle_o \} \geq \min\{ \langle \alpha_1^m,\alpha_1^n \rangle_o, 1+n\} .\]
Therefore,
\[ \sup_m \liminf_{k,n\to \infty} \langle \alpha_1^m,\alpha_k^n \rangle_o \succeq \sup_m \liminf_{n\to \infty} \langle \alpha_1^m,\alpha_1^n \rangle_o = \infty,\]
from which it follows that $\calG(\calL_n) \to \calG(\calL)$, as required.  This completes the proof.
\end{proof}

\begin{proof}[Proof of Theorem~\ref{survivalendinglam}]
Let $\calF \colon \partial \C^s(\dot S) \to \EL^s(\dot S)$ be the homeomorphism from Theorem~\ref{T:boundary ending precise}.  It suffices to show that $\calF$ is $\PMod(\dot S)$--equivariant.  For this, let $f \in \PMod(\dot S)$ be any mapping class and $x \in \partial \C^s(\dot S)$ any boundary point.  If $\{\alpha_n\} \subset \C^s(\dot S)$ is any sequence with $\alpha_n \to x$ in $\bar \C^s(\dot S)$, then $f \cdot \alpha_n \to f \cdot x$ since $f$ acts by isometries on $\C^s(\dot S)$.    Applying Theorem~\ref{T:boundary ending precise} to the sequence $\{f \cdot \alpha_n\}$ we see that $f \cdot \alpha_n \stackrel{CH}\to \calF(f \cdot x)$.  On the other hand we also have $f \cdot \alpha_n \stackrel{CH}\to f \cdot \calF(x)$, since $f$ acts by homeomorphisms on the space of laminations with the coarse Hausdorff topology.  Therefore, $f \cdot \calF(x) = \calF(f \cdot x)$, as required.
\end{proof}
\section{Extended survival map} \label{S:extended survival} $\quad$
We start by introducing some notation before we define the
extended survival map.  First observe that there is an injection $\C(S) \to \PML(S)$ given by sending a point in the interior of the simplex $\{v_0,\ldots,v_k\}$ with barycentric coordinates $(s_0,\ldots,s_k)$ to the projective class, $[s_0v_0 + \cdots + s_kv_k]$; here we are viewing $s_0v_0 + \cdots s_kv_k$ as a measured geodesic lamination with support $v_0 \cup \ldots \cup v_k$ and with the transverse counting measure scaled by $s_i$ on the $i^{th}$ component, for each $i$.  We denote the image by $\PML_{\C}(S)$, which by construction admits a bijective map $\PML_\C(S) \to \C(S)$ (inverse to the inclusion above).

By Theorem~\ref{T:Klarreich}, $\partial \calC(S) \cong \EL(S)$, and so it is natural to define
\[ \PML_{\bar\calC}(S)=\PML_{\calC}(S) \cup \PFL(S), \]
and we extend the bijection $\PML_\C(S) \to \C(S)$ to a surjective map
\[ \PML_{\overline\calC}(S) \rightarrow \overline \calC(S)\]
By Proposition~\ref{P:support continuous} and Theorem~\ref{T:Klarreich}, this is continuous at every point of $\PFL(S)$.

Similar to the survival map $\tilde \Phi$ defined in Section~\ref{S:tree map construction}, we can define a map
\[ \widetilde \Psi \colon \PML(S) \times \Diff_0(S) \to \PML(\dot S). \]
This is defined by exactly the same procedure as in Section~2.4 of \cite{LeinMjSch}, which goes roughly as follows: If $\mu$ is a measured lamination with no closed leaves in its support $|\mu|$, and if $f(z) \not\in |\mu|$, then $\widetilde \Psi(\mu,f) = f^{-1}(\mu)$.  When $|\mu|$ contains closed leaves we replace those with the foliated annular neighborhoods of such curves defined in Section~\ref{S:tree map construction}).  When the $f(z)$ lies on a leaf of $|\mu|$ (or the modified $|\mu|$ when there are closed leaves) we ``split $|\mu|$ apart at $f(z)$", then take the $f^{-1}$--image.  The same proof as that given in \cite[Proposition~2.9]{LeinMjSch} shows that $\widetilde \Psi$ is continuous.


As in Section~\ref{S:tree map construction} (and in \cite{LeinMjSch}) via the lifted evaluation map $\widetilde{\ev} \colon \Diff_0(S) \to \mathbb H$, given by $\widetilde{\ev}(f) = \tilde f(\tilde z)$ (for $\tilde f$ the canonical lift), the map $\widetilde \Psi$ descends to a continuous, $\pi_1S$--equivariant map $\Psi$ making the following diagram commute:
\begin{center}
\begin{tikzpicture}
    \node (E) at (0,0) {$\PML(S)\times \Diff_0(S)$};
    \node[below=of E] (B) {$\PML(S)\times \mathbb{H}$};
    \node[right=of B] (A) {$\PML(\dot S)$.};
    \draw[->] (E)--(A) node [midway, right,above] {$\widetilde\Psi$};
    \draw[->] (B)--(A) node [midway,above] {$\Psi$};
    \draw[->] (E)--(B) node [midway,left] {$\id_{\PML(S)} \times \widetilde{\ev}$};
\end{tikzpicture}
\end{center}

By construction, the restriction $\Psi_\C = \Psi|_{\PML_\C(S) \times \mathbb H}$ and $\Phi$ agree after composing with the bijection between $\PML_\C(S)$ and $\C(S)$ in the first factor.  Since $\Phi$ maps $\C(S) \times \mathbb H$ onto $\C^s(\dot S)$, if we define $\PML_{\C^s}(\dot S)$ to be the image of $\C^s(\dot S)$ via the map $\C^s(\dot S) \to \PML(\dot S)$ defined similarly to the one above, then the following diagram of $\pi_1S$--equivariant maps commutes, with the vertical arrows being bijections
\begin{equation} \label{E:Psi C}
\xymatrixcolsep{4pc}\xymatrixrowsep{3pc}\xymatrix{
\PML_\C(S) \times \mathbb H \ar[r]^{\quad \Psi_\C}  \ar[d] & \PML_{\C^s}(\dot S) \ar[d] \\
\C(S) \times \mathbb H \ar[r]^{\quad \Phi} & \C^s(\dot S) }
\end{equation}

Similar to $\PML_{\bar \C}(S) = \PML_\C(S) \cup \PFL(S)$ above, we define
\[ \PML_{\bar \C^s}(\dot S) = \PML_{\C^s}(\dot S) \cup \PFL^s(\dot S),\]
where, recall, $\PFL^s(\dot S)$ is the space of measured laminations on $\dot S$ whose support is contained in $\EL^s(\dot S)$.   Then $\Psi_\C$ extends to a map
\[ \Psi_{\bar \C} \colon \PML_{\bar \C}(S) \times \mathbb H \to \PML_{\bar \C^s}(\dot S).\]
The fact that $\Psi([\mu],w)$ is in $\PFL^s(\dot S)$ for any $w \in \mathbb H$ and $[\mu] \in \PFL(S)$ is straightforward from the definition (c.f.~\cite[Proposition~2.12]{LeinMjSch}): for {\em generic} $w$, $\Psi([\mu],w)$ is obtained from $[\mu]$ by adding the $z$--puncture in one of the complementary components of $|\mu|$ and adjusting by a homeomorphism.  With this, it follows that the map $\Phi$ extends to a map $\hat \Phi$ making the following diagram, extending \eqref{E:Psi C}, commute.
\begin{equation*}
\xymatrixcolsep{4pc}\xymatrixrowsep{3pc}\xymatrix{
\PML_{\bar \C}(S) \times \mathbb H \ar[r]^{\quad \Psi_{\bar\C}}  \ar[d] & \PML_{\bar \C^s}(\dot S) \ar[d] \\
\bar\C(S) \times \mathbb H \ar[r]^{\quad \hat \Phi} &\bar \C^s(\dot S) }
\end{equation*}


We will call the map $\hat\Phi:\bar{\C}(S)\times \mathbb{H} \rightarrow \csC $ the \textit{extended survival map}.
Vertical maps in the diagram are natural maps which take projective measured laminations to their supports and they send $\PFL(S)\times \mathbb{H}$  onto $\EL(S)\times \mathbb{H}$ and $\PFL^s(\dot S)$ onto $\EL^s(\dot S)$.
\begin{lemma} \label{L:continuity of hat Phi} The extended survivial map $\hat \Phi$ is $\pi_1S$--equivariant and is continuous at every point of $\partial\calC(S)$.
\end{lemma}
\begin{proof}
To prove the continuity statement, we use the homeomorphism $\mathcal F$ from Theorem~\ref{T:boundary ending precise} to identity $\partial \C^s(\dot S)$ with $\EL^s(\dot S)$.  Now suppose $\{\calL_n\}\in \overline\calC(S)$, $\calL\in \partial \calC(S)$ , $\calL_n \rightarrow \calL$ and $\{x_n\}$ be a sequence in $\mathbb{H}$ such that $x_n \rightarrow x$. Passing to a subsequence, there is a measure $\mu_n$ on $\calL_n$ and a measure $\mu$ on $\calL$ such that $\mu_n \rightarrow \mu$ in $\ML(S)$. Since $\Psi$ is continuous on $\PFL(S)\times \mathbb{H}$
\[ \Psi([\mu_n], x_n) \rightarrow  \Psi([\mu], x).    \]
By Proposition~\ref{P:support continuous} this implies,
\[| \Psi(\mu_n, x_n)| \stackrel{CH}{\longrightarrow} |\Psi(\mu, x)|.  \]
On the other hand,  $\Psi(\FL(S)\times \mathbb{H})\subset \FL^s(\dot S)$, and by Theorem \ref{T:boundary ending precise} this means
\[ \hat \Phi(\calL_n,x_n) \to \hat \Phi(\calL,x) \]
in $\bar \C^s(\dot S)$, since $|\hat \Psi(\mu_n, x_n)|= \hat \Phi(\calL_n, x_n )$ and  $|\hat \Psi(\mu, x)|=\hat\Phi (\calL,x)$.

The $\pi_1S$--equivariance follows from that of $\Phi$ on $\C^s(\dot S)$ and continuity at the remaining points.
\end{proof}

The following useful fact and it's proof are identical to the statement and proof of \cite[Lemma~2.14]{LeinMjSch}.
\begin{lemma} \label{L:points identified by hat phi} Fix $(\calL_1,x_1),(\calL_2,x_2) \in \EL(S) \times \HH$.  Then $\hat \Phi(\calL_1,x_1) = \hat \Phi(\calL_2,x_2)$ if and only if $\calL_1 =\calL_2$ and $x_1,x_2$ are on the same leaf of, or in the same complementary region of, $p^{-1}(\calL_1) \subset \HH$.
\end{lemma}

Suppose that $\calL \in \EL(S)$, $x \in \calP \subset \partial \mathbb H$ is a parabolic fixed point, $H_x \subset \mathbb H$ is the horoball based at $x$ as in Section~\ref{S:cusps and witnesses}, and $U \subset \mathbb H$ is the complementary region of $p^{-1}(\calL)$ containing $H_x$.  Given $y  \in U$, choose any $f \in \Diff(S)$ so that $\tilde f(\tilde z) = y$, so that $\hat \Phi(\calL,y) = f^{-1}(\calL)$.  Observe that $p(U)$ is a complementary region of $\calL$ containing a puncture (corresponding to $x$), and hence $\hat \Phi(\calL,y)$ is a lamination with two punctures in the complementary component $f^{-1}(p(U))$ (one of which is the $z$-puncture).  Therefore, $\hat \Phi(\calL,y)$ is an ending lamination in a proper witness.
More precisely, by Lemma~\ref{L:points identified by hat phi}, we may assume $y \in H_x$ without changing the image $\hat \Phi(\calL,y)$, and then as in the proof of Lemma~\ref{L:W parabolics}, $f^{-1}(\partial p(H_x))$ is the boundary of the witness $\calW(x)$ which is disjoint from $\hat \Phi(\calL,x)$. Thus, $\hat \Phi(\calL,y) \in \EL(\calW(x))$.

In fact, every ending lamination on a proper witness arises as such an image as the next lemma shows.

\begin{lemma} \label{L:where the witness ELs come from} Suppose $\calL_0 \in \EL(W)$ is an ending lamination in a proper witness $W \subsetneq \dot S$.  Then there exists $\calL \in \EL(S)$, $x \in \calP$, a complementary region $U$ of $p^{-1}(\calL)$ containing $H_x$, and $y \in H_x$ so that $\calW(x)=W$ and $\hat \Phi(\calL,y) = \calL_0$.
\end{lemma}
\begin{proof} Note that the inclusion of $W \subset \dot S$ is homotopic through embeddings to a diffeomorphism, after filling in $z$ (since after filling in $z$, $\partial W$ is peripheral).  Consequently, after filling in $z$, $\calL_0$ is isotopic to a geodesic ending lamination $\calL$ on $S$.  Let $f \colon S \to S$ be a diffeomorphism isotopic to the identity with $f(\calL_0) = \calL$.  Then $\calL_0 = f^{-1}(\calL) = \widetilde \Psi([\mu],f)$ where $[\mu]$ is the projective class of any transverse measure on $\calL$.

Next, observe that $f(z)$ lies in a complementary region $V$ of $\calL$ which is a punctured polygon (since $\partial W$ is a simple closed curve disjoint from $\calL_0$ bounding a twice punctured disk including the $z$-puncture).  Let $U \subset \mathbb H$ be the complementary region of $p^{-1}(\calL)$ that projects to $V$.  Then $U$ is an infinite sided polygon invariant by a parabolic subgroup fixing some $x \in \calP$.
Now let $\tilde f \colon \mathbb H \to \mathbb H$ be the canonical lift as in Section~\ref{S:tree map construction} and let $y' = \tilde f(\tilde z)$, so that by definition $\widetilde \Psi([\mu],f) = \Psi([\mu],y') = \hat \Phi(\calL,y')$.  By Lemma~\ref{L:points identified by hat phi}, for any $y \in H_x \subset U$, it follows that $\hat \Phi(\calL,y) = \hat \Phi(\calL,y') = \calL_0$.  From the remarks preceeding this lemma, it follows that $\calL_0 \in \EL(\calW(x))$.  Since $\EL(W) \cap \EL(W') = \emptyset$, unless $W = W'$, it follows that $\calW(x) = W$, completing the proof.
\end{proof}

\section{Universal Cannon--Thurston maps} \label{S:UCT maps}

In this section we will prove the following.

\bigskip

\noindent {\bf Theorem~\ref{CT}}{\em \CTstatement}

\bigskip

Before proceeding, we describe the subset $\Abd \subset \partial \mathbb H^2$ in Theorem~\ref{CT}.

\begin{definition} Let $Y \subseteq \dot S$ be a subsurface.  A point $x\in \partial\HH$ \textit{fills} $Y$ if,
\begin{itemize}
\item The image of every geodesic ending in $x$ projected to $\dot S$ intersects every curve which projects to $Y$,
\item There is a geodesic ray $r\subset \HH$ ending at $x$ with $p(r)\subset Y$.
\end{itemize}
Now let $\Abd\subset \partial \HH$ be the set of points that fill $\dot S$.
\end{definition}

We note that when $x \not \in \Abd$, there is a ray $r$ ending at $x$ so that $r$ is contained in a proper subsurface $Y \subsetneq \dot S$.  The boundary of this subsurface is an essential curve $v$ and $\Phi_v(r) \subset T_v$ is a bounded diameter set.  Thus, restricting to the set $\Abd$ is necessary (c.f.~\cite[Lemma~3.4]{LeinMjSch}).

Given the modifications to the setup, the existence of the extension of Theorem~\ref{CT} follows just as in the case that $S$ is closed in \cite{LeinMjSch}; this is outlined in Section~\ref{S:existence}. The surjectivity requires more substantial modification, however, and this is carried out in Section~\ref{S:surjectivity}.  The proof of the universal property of $\partial \Phi$, as well as the discussion of $\partial \Phi_0 \colon \partial \C(S) \to \partial \C(\dot S)$, Theorem~\ref{T:UCT C short}, and the relationship to Theorem~\ref{CT} is carried out in Section~\ref{S:Universal}.

\subsection{Quasiconvex nesting and existence of Cannon--Thurston maps} \label{S:existence}

In this section we will prove the existence part of the Theorem \ref{CT}.

\begin{theorem}\label{CTex} For any vertex $v\in \calC(S)$, the induced survival map $\Phi_v \colon \HH \rightarrow \sC$ has a continuous, $\pi_1(S)$--equivariant extension to
\[\overline\Phi_v: \HH\cup \Abd\rightarrow \overline\C^s(\dot S))\]
Moreover, the restriction $\partial \Phi_v = \overline{\Phi}_v|_{\Abd} \colon \Abd \to \partial \C^s(\dot S)$ does not depend on the choice of $v$.
\end{theorem}
By the last claim, we may denote the restriction as $\partial \Phi  \colon \Abd \to \partial \C^s(\dot S)$, without reference to the choice of $v$.
To prove this theorem, we will use the following from \cite[Lemma~1.9]{LeinMjSch}, which is a mild generalization of a lemma of Mitra in \cite{Mitra1}.
\begin{lemma} \label{mitra} Let $X$ and $Y$ be two hyperbolic metric spaces, and $F: X \rightarrow Y$ a continuous map. Fix a basepoint $y\in Y$ and a subset $A \subset \partial X$. Then there is a $A$--Cannon-Thurston map
\[\overline F: X\cup A \rightarrow Y\cup \partial Y\]
if and only if for all $s\in A$ there is a neighborhood basis $\calB_i\subset X \cup A$ of $s$ and a collection of uniformly quasiconvex sets $Q_i\subset Y$ such that;
\begin{itemize}
\item $F(\calB_i\cap X )\subset Q_i$, and
\item $d_Y(y, Q_i)\rightarrow \infty$ as $i\rightarrow \infty$.
\end{itemize}
Moreover,
\[\bigcap_{i}\overline Q_i=\bigcap_{i}\partial Q_i =\{\overline F(s)\}\]
determines $\overline F(s)$ uniquely, where $\partial Q_i = \bar Q_i \cap \partial Y$.
\end{lemma}

\noindent Given the adjustments already made to our setup, the proof of Theorem \ref{CTex} is now nearly identical to \cite[Theorem~3.6]{LeinMjSch}, so we just recall the main ingredients, and explain the modifications necessary in our setting.\\

\noindent We fix a bi-infinite geodesic  $\gamma$ in $\HH$ so that $p(\gamma)$ is a closed geodesic that fills $S$ (i.e.~nontrivially intersects every essential simple closed curve or arc on $S$).
As in \cite{LeinMjSch}, we construct quasi-convex sets from such $\gamma$ as follows. Define
\[\calX(\gamma)=\Phi(\C(S)\times \gamma)\]
where $\Phi$ is the survival map. Let $\calH^{\pm}(\gamma)$ denote the two half spaces bounded by $\gamma$ and define the sets
\[ \mathscr{H}^{\pm}(\gamma)= \Phi(\C(S)\times \calH^{\pm}(\gamma))     \]

\noindent The proofs of the following two facts about these sets are identical to the quoted results in \cite{LeinMjSch}.
\begin{itemize}
\item \cite[Proposition~3.1]{LeinMjSch}: $\calX(\gamma)$, $ \mathscr{H}^{\pm}(\gamma)$ are  simplicial subcomplexes of $\sC$ spanned by their vertex sets and are weakly convex (meaning every two points in the set are joined by {\em some} geodesic contained in the set).
\item \cite[Proposition~3.2]{LeinMjSch}: We have,
\[  \mathscr{H}^{+}(\gamma)\cup  \mathscr{H}^{-}(\gamma)= \sC\]
and
\[  {\mathscr{H}}^{+}(\gamma)\cap  \mathscr{H}^{-}(\gamma) = \calX(\gamma).\]

\end{itemize}

Now we consider a set $\{\gamma_n\}$ of pairwise disjoint translates of $\gamma$ in $\HH$ so that the corresponding (closed) half spaces nest:
\[ \calH^{+}(\gamma_1) \supset\calH^{+}(\gamma_2)\supset\cdots\]
\noindent Since the action is properly discontinuously on $\HH$, there is a $x\in \partial \HH$ such that
\begin{equation}\label{eq:1}
  \bigcap^{\infty}_{n=1}\overline{\calH^{+}(\gamma_n)}=\{x\}.
\end{equation}
Here, $\overline{\calH^+(\gamma_n)}$ is the closure in $\overline{\mathbb H}$.  For such a sequence, we say {\em $\{\gamma_n\}$ nests down on $x$}.

On the other hand, if $r \subset \mathbb H$ is a geodesic ray ending in some point $x\in \partial \mathbb H$ which is {\em not} a parabolic fixed point, $p(r)$ intersects $p(\gamma)$ infinitely many times.  Hence, we can find a sequence $\{\gamma_n\}$ which nests down on $x$.  In particular, for any element $x \in \Abd$ has a sequence $\{\gamma_k\}$ that nests down on $x$.

The main ingredient in the proof of existence of the extension is the following.

\begin{proposition}\label{nestinglimit} If $\{\gamma_n\}$  nests down to $x\in \Abd $, then for a basepoint $b\in \sC$, the sets $\mathscr{H}^{+}(\gamma_n)$ are quasiconvex and  we have
\[d^s(b, \mathscr{H}^{+}(\gamma_n))\rightarrow \infty \,\,\text{as}\,\, n\rightarrow \infty\]
\end{proposition}
The proof is nearly identical to that of \cite[Proposition~3.5]{LeinMjSch}, but since it's the key to the proof of existence, we sketch it for completeness.
\begin{proof}[Sketch of proof] Because of the nesting in $\HH$, we have nesting in $\C^s(\dot S)$,
\[ \mathscr{H}^+(\gamma_1) \supset \mathscr{H}^+(\gamma_2) \supset \cdots . \]
We must show that for any $R>0$, there exists $N >0$ so that $d^s(b,\mathscr{H}^+(\gamma_n)) \geq R$, for all $n \geq N$.  The first observation is that because $\Pi \colon \C^s(\dot S) \to \C(S)$ is simplicial (hence $1$--Lipschitz), it suffices to find $N >0$ so that $d^s(b,\mathscr{H}^+(\gamma_n) \cap \Pi^{-1}(B_R(\Pi(b))) \geq R$ for all $n \geq N$.

To prove this, one can use an inductive argument to construct an  increasing sequence $N_1 < N_2 < \ldots <N_{R+1}$ so that
\[ \calX(\gamma_{N_j}) \cap \Pi^{-1}(B_R(\Pi(b)))  \cap \calX(\gamma_{N_{j+1}}) = \emptyset.\]
Before explaining the idea, we note that this implies that $\{\mathscr{H}^+(\gamma_{N_j}) \cap \Pi^{-1}(B_R(\Pi(b)))\}_{j=1}^{R+1}$ are {\em properly} nested: a path from $b$ to $\mathscr{H}^+(\gamma_{N_{R+1}})$ inside $\Pi^{-1}(B_R(\Pi(b)))$ must pass through a vertex of $\mathscr{H}^+(\gamma_{N_j})$, for each $j$,  before entering the next set.  Therefore, it must contain at least $R+1$ vertices, and so have length at least $R$.  This completes the proof by taking $N = N_{R+1}$, since then a geodesic from $b$ to a point of $\mathscr{H}^+(\gamma_{N_{R+1}})$ will have length at least $R$ (if it leaves $\Pi^{-1}(B_R(\Pi(b)))$, then it's length is greater than $R$).

The main idea to find the sequence $N_1 < N_2 < \ldots N_{R+1}$ is involved in the inductive step.  If we have already found $N_1 < N_2 < \ldots N_{k-1}$, and we want to find $N_k$, we suppose there is no such $N_k$, and derive a contradiction.  For this, assume
\[ \calX(\gamma_{N_{k-1}}) \cap \calX(\gamma_n) \cap \Pi^{-1}(B_R(\Pi(b))) \neq \emptyset,\]
for all $n > N_{k-1}$, and let $u_n$ be a vertex in this intersection.  Set $v_n = \Pi(u_n)$, and recall that $\Phi_{v_n}^{-1}(u_n) = U_n \subset \HH$ is a component of the complement a small neighborhood of the preimage in $\HH$ of the geodesic representative of $v_n$ in $S$.  The fact that $u_n \in \calX(\gamma_{N_{k-1}}) \cap \calX(\gamma_n)$ translates into the fact that $\gamma_{N_{k-1}} \cap U_n \neq \emptyset$ and $\gamma_n \cap U_n \neq \emptyset$.  After passing to subsequences and extracting a limit, we find a geodesic from a point on $\gamma_{N_{k-1}}$ (or one of its endpoints in $\partial \HH$) to $x$, which projects to have empty transverse intersection with $v_n$ in $S$. Since $v_n$ is contained in the bounded set $B_R(\Pi(b))$, any subsequential Hausdorff limit does not contain an ending lamination on $S$, by Theorem~\ref{T:Klarreich}, and so any ray with no transverse intersections is eventually trapped in a subsurface (a component of the minimal subsurface of the maximal measurable sublamination of the Hausdorff limit).  This contradicts the fact that $x \in \Abd$, and completes the sketch of the proof.
\end{proof}

We are now ready for the proof of the existence part of Theorem~\ref{CT}.
\begin{proof}[Proof of Theorem \ref{CTex}]  The existence and continuity of $\overline\Phi_v$ follows by verifying the hypotheses in Lemma \ref{mitra}.

Fix a basepoint  $b\in \sC$ and let $\{\gamma_n\}$ be a sequence nesting to a point $x\in \Abd$. The collection of sets
\[ \{ \overline{\calH^{+}(\gamma_n)} \cap (\mathbb H \cup \Abd)\}_{n=1}^\infty, \]
is a neighborhood basis of $x$ in $\mathbb H \cup \Abd$.
By definition of $\mathscr H^+(\gamma_n)$
\[\Phi_v(\calH^{+}(\gamma_n))=\Phi (\{v\}\times \calH^{+}(\gamma_n)) \subset \mathscr{H}^{+}(\gamma_n),\]
for all $n$.
By Proposition \ref{nestinglimit}, $d^s(b, \mathscr{H}^{+}(\gamma_n))\rightarrow \infty \,\,\text{as}\,\, n\rightarrow \infty$.
Therefore, by Lemma \ref{mitra} we have a $\Abd$-Cannon--Thurston map $\overline\Phi_v$ defined on $x \in \Abd$ by
\[ \{\overline\Phi_v(x)\} =  \bigcap_{n=1}^\infty \overline{\mathscr{H}^+(\gamma_n)}.\]
Since the sets on the right-hand side do not depend on the choice of $v$, and since $x \in \Abd$, we also write $\partial \Phi(x) = \overline\Phi_v(x)$, and note that $\partial \Phi \colon \Abd \to \partial \C^s(\dot S)$ does not depend on $v$.
\end{proof}

Observe that for all $x \in \Abd$, we have
\begin{equation} \label{E:what is partial Phi} \partial \Phi (x) = \bigcap_{n=1}^\infty \partial \mathscr{H}^+(\gamma_n)
\end{equation}
where $\{\gamma_n\}$ is any sequence nesting down on $x$, because the intersection of the closure is in fact the intersection of the boundaries.

\subsection{Surjectivity of the Cannon-Thurston map} \label{S:surjectivity}

We start with the following lemma.

\begin{lemma}\label{surj1}For any $v\in \calC^0(S)$ we have
 \[\partial\sC \subset\overline {\Phi_v(\mathbb{H})}\]
\end{lemma}
The analogous statement for $S$ closed is \cite[Lemma~3.12]{LeinMjSch}, but the proof there does not work in our setting. Specifically, the proof in \cite{LeinMjSch} appeals to Klarreich's theorem about the map from Teichm\"uller space to the curve complex, and extension to the boundary of that; see \cite{Klarreich}.  In our situation, the analogue would be a map from Teichm\"uller space to $\C^s(\dot S)$, to which Klarreich's result does not apply.

\begin{proof} We first claim that if $X\subset \partial\sC$ is closed and $\PMod(\dot S)$--invariant then either $X=\emptyset$ or $X=\partial\sC$. This is true since the set ${\rm{PA}}$ of fixed points of pseudo-Anosov elements of $\PMod(\dot S)$ is dense in $\EL(\dot{S})$ and $\EL(\dot{S})$  is dense in $\EL^s(\dot S)$.  As a result, ${\rm{PA}}$ is dense in $\partial\sC$. Since any nonempty, closed, pure mapping class group invariant subset of $\partial \sC$ has to include ${\rm{PA}}$, the claim is true.

Now we will show that $\partial\sC \cap \overline{\Phi_v(\mathbb{H})}$ contains a $\PMod(\dot S)$--invariant set.  For this, first let ${\rm{PA}}_0 \subset {\rm{PA}}$ be the set of pseudo-Anosov fixed points for elements in $\pi_1S < \PMod(\dot S)$.  Since the $\pi_1(S)$ action leaves $\Phi_v(\mathbb H)$ invariant, and since pseudo-Anosov elements act with north-south dynamics on $\overline{\mathcal C}^s(\dot S)$, it follows that ${\rm{PA}}_0 \subset \overline{\Phi_v(\mathbb H)}$.
Next, we need to show that $f({\rm{PA}}_0)={\rm{PA}}_0$ for $f\in \PMod(\dot S)$.  For any point $x \in PA_0$, let $\gamma \in \pi_1(S)$ be a pseudo-Anosov element with $\gamma(x) = x$.  Then $f \gamma f^{-1}$ fixes $f(x)$, but  $f \gamma f^{-1}$ is also a pseudo-Anosov element of $\pi_1(S)$, since $\pi_1(S)$ is a normal subgroup of $\PMod(\dot S)$.   So, $f({\rm{PA}}_0) \subset {\rm{PA}}_0$, since $x \in {\rm{PA}}_0$ was arbitrary.  Applying the same argument to $f^{-1}$, we find $f({\rm{PA}}_0) = {\rm{PA}}_0$.  Since $f \in \PMod(\dot S)$ was arbitrary, ${\rm{PA}}_0$ is $\PMod(\dot S)$--invariant.

Therefore, $\overline {\rm{PA}}_0$ is a nonempty closed $\PMod(\dot S)$--invariant subset of $\partial \sC \cap \overline{\Phi_v(\mathbb H)}$, and so both of these sets equal $\partial \sC$.
\end{proof}
To prove the surjectivity, we will need the following proposition.  The exact analogue for $S$ closed is much simpler, but is false in our case (as the second condition suggests); see \cite[Proposition~3.13]{LeinMjSch}.  To state the proposition, recall that $\mathcal P \subset \HH$ denotes the set of parabolic fixed points; see Section~\ref{S:cusps and witnesses}.

\begin{proposition}\label{surj2} If $\{x_n\}$ is a sequence of points in $\mathbb{H}$ with limit $x\in \partial\mathbb{H} \setminus \Abd$, then one of the following holds:
\begin{enumerate}
\item $\Phi_v(x_n)$ does not converge to a point of $\partial\sC $; or
\item $x \in \mathcal P$ and $\Phi_v(x_n)$ accumulates only on points in $\partial \C(\mathcal W(x))$.
\end{enumerate}
\end{proposition}

To prove this, we will need the following lemma.  For the remainder of this paper, we identify the points of $\partial \C^s(\dot S)$ with $\EL^s(\dot S)$ via Theorem~\ref{T:boundary ending precise}.

\begin{lemma} \label{L:a new annulus} Suppose $x \in \mathcal P$ and $\{x_n\} \subset \mathbb H$ with $x_n \to x$.  If $\Phi_v(x_n) \to \calL$ in $\partial \sC$, then $\calL \in \partial \C(\mathcal W(x))$.
\end{lemma}
\begin{proof}   We suppose $\Phi_v(x_n) \to \calL$ in $\partial \sC$.
Let $H = H_x \subset \mathbb H$ be the horoball based at $x$ disjoint from all chosen neighborhoods of geodesics used to define $\Phi$ as in Sections~\ref{S:tree map construction} and \ref{S:cusps and witnesses}.  Applying an isometry if necessary, we can assume that $x = \infty$ in the upper-half plane model and $H = \{z \in \mathbb C \mid {\rm{Im}}(z) \geq 1 \}$ is stabilized by the cyclic, parabolic group $\langle g \rangle < \pi_1(S,z)$.  By Lemma~\ref{L:W parabolics}, the $\Phi_v$--image of $H$ is a single point $\Phi_v(H)= \{u\}$.  Note that if ${\rm{Im}}(x_n) > \epsilon > 0$ for some $\epsilon > 0$, then $\Phi_v(x_n)$ remains a bounded distance from $u$, and hence does not converge to any $\calL \in \partial \sC$.  Therefore, it must be that ${\rm{Im}}(x_n) \to 0$ and consequently ${\rm{Re}}(x_n) \to \pm \infty$.

We may pass to a subsequence so that the hyperbolic geodesics $[x_n,x_{n+1}]$ nontrivially intersect $H$, and from this find a sequence of points $y_n \in H$ and curves $v_n \in \C(S)$ so that $u_n = \Phi(v_n,y_n)  \to \calL$ as $n \to \infty$ (c.f.~the proof of \cite[Proposition 3.11]{LeinMjSch}).  According to Lemma~\ref{L:W parabolics}, $\Phi(\C(S) \times H) = \C(\calW(x))$, and so $\Phi(v_n,y_n) \in \C(\calW(x))$.  Consequently, $\calL \in \partial \C(\calW(x))$, as required.
\end{proof}

\begin{proof}[Proof of Proposition~\ref{surj2}] We suppose $\Phi_v(x_n) \to \calL \in \partial \sC$ and argue as in \cite{LeinMjSch}.  Specifically, the assumption that $x \in \partial \HH \setminus \Abd$ means that a ray $r$ ending at $x$, after projecting to $S$, is eventually trapped in some proper, $\pi_1$--injective subsurface $Y \subset S$, and fills $Y$ if $Y$ is not an annulus.  If $Y$ is not an annular neighborhood of a puncture, then we arrive at the same contradiction from \cite[Proposition~3.13]{LeinMjSch}.  On the other hand, if $Y$ is an annular neighborhood of a puncture, then by Lemma~\ref{L:a new annulus}, $\calL \in \partial \C(\mathcal W(x))$, as required.
\end{proof}

\begin{theorem} The Cannon--Thurston map
\[ \partial \Phi: \Abd \rightarrow \partial \sC\]
is surjective and $\PMod(\dot S)$--equivariant.
\end{theorem}
\begin{proof} Let $\calL \in  \partial \sC$. Then, by Lemma \ref{surj1} $\calL=\lim \Phi_v(x_n)$
for some sequence $\{x_n\}\in \mathbb{H}$. Passing to a subsequence, assume that $x_n \rightarrow x$ in $\partial\mathbb{H}$. If $x\in\Abd $ we are done since by continuity at every point of $\Abd$ we have,
 \[\calL=\lim \Phi_v(x_n)=\overline\Phi_v(x)=\partial\Phi(x). \]
 If  $x\notin \Abd$,  then by Proposition~\ref{surj2}, $x\in \calP$ and $\calL \in \partial \C(W)$, where $W = \calW(x)$.   By Lemma~\ref{L:proj and conv2}, $\pi_W(\Phi_v(x_n)) \to \calL \in \partial \C(W)$.  Let $g\in \pi_1(\dot S)$ be the generator of $\stab_{\pi_1S}(x)$.   As in the proof of Lemma~\ref{L:a new annulus}, $x_n$ is not entirely contained in any horoball based at $x$, and hence it must be that there exists a sequence $\{k_n\}$ such that $g^{k_n}(x_n)\rightarrow \xi$  where $\xi\in \partial\mathbb H$ is some point such that $\xi\neq x$.  Since $g$ is a Dehn twist in $\partial W$, it does not affect $\pi_W(\Phi_v(x_n))$.  Thus $\pi_W(\Phi_v(g^{k_n}(x_n))) \to \calL$ and hence $\Phi_v(g^{k_n}(x_n)) \to \calL$ by another application of Lemma~\ref{L:proj and conv2}.
 But in this case since $\xi\neq x$ does not satisfy either of the possibilities given in Lemma \ref{surj2}, and hence $\xi\in  \Abd$.  But this implies
 \[ \calL = \lim_{n \to \infty} \Phi_v(g^{k_n}(x_n)) = \overline{\Phi}_v(\xi),\]
 again appealing to continuity of $\overline{\Phi}$.

The proof of $\PMod(\dot S)$--equivariance is identical to the proof of \cite[Theorem~1.2]{LeinMjSch}.  The idea is to use $\pi_1S$--equivariance, and prove $\partial \Phi(\phi \cdot x) = \phi \cdot \partial \Phi(x)$ for $\phi \in \PMod(\dot S)$ and $x$ in the dense subset of $\Abd$ consisting of attracting fixed points of elements $\delta \in \pi_1S$ whose axes project to filling closed geodesics on $S$.  The point is that such points $x$ are attracting fixed points in $\partial\HH$ of $\delta$, but their images are also attracting fixed points in $\partial \C^s(\dot S)$ since $\delta$ is pseudo-Anosov by Kra's Theorem \cite{Kra}, when viewed as an element of $\PMod(\dot S)$.  The fact that $\phi(x)$ and $\phi(\partial \Phi(x))$ are the attracting fixed points of $\phi \delta \phi^{-1}$ in $\partial \HH$ and $\partial \C^s(\dot S)$, respectively, finishes the proof.
\end{proof}

\subsection{Universality and the curve complex} \label{S:Universal}
The following theorem on the universality is an analog of \cite[Corollary3.10]{LeinMjSch}.  While the statement is similar, it should be noted that in \cite{LeinMjSch}, the map is finite-to-one, though this is not the case here since some of the complementary regions of the preimage in $\HH$ of laminations in $S$ are infinite sided ideal polygons, and whose sides accumulate to a parabolic fixed point.  We follow \cite{LeinMjSch} where possible, and describe the differences when necessary.

\begin{CTUtheorem} \label{T:CU theorem}
Given two distinct points $x,y\in \Abd$, $\partial\Phi(x) = \partial \Phi(y)$ if and only if $x$ and $y$ are the ideal endpoints of a leaf or complementary region of $p^{-1}(\calL)$ for some $\calL \in \EL(S)$.
\end{CTUtheorem}


The proof will require a few additional facts.  The first is the analogue of \cite[Proposition~3.8]{LeinMjSch} which states that the intersections at infinity of the images of the half-spaces satisfy
\begin{equation} \label{E:boundary intersections 1}
\partial \mathscr{H}^+(\gamma) \cap \partial \mathscr{H}^-(\gamma) = \partial \calX (\gamma)
\end{equation}
where as above, $\gamma$ is a geodesics that projects to a closed, filling geodesic in $S$.
The next is the analogue in our setting of \cite[Lemma~3.9]{LeinMjSch}.  To describe this, recall that the element $\delta \in\pi_1S$ stabilizing $\gamma$ is a pseudo-Anosov mapping class when viewed in $\PMod(\dot S)$ by a theorem of Kra \cite{Kra}. Let $\pm \calL \in \EL(\dot S) \subset \partial \C^s(\dot S)$ be the attracting and repelling fixed points (i.e. the stable/unstable laminations).  Then we have
\begin{equation} \label{E:boundary intersections 2}
\partial \calX(\gamma) =  \hat \Phi(\partial \C(S) \times \gamma) \cup \{ \pm \calL \}
\end{equation}
The proofs of these facts are identical to those in \cite{LeinMjSch}, and we do not repeat them.




\begin{proof}[Proof of Theorem~\ref{CTU}]  Given $x,y \in \Abd$, first suppose that there is an ending lamination $\calL \in \EL(S)$ and $E \subset \HH$ which is either a leaf or complementary region of $p^{-1}(\calL)$, so that $x$ and $y$ are ideal vertices of $E$.  Let $\{\gamma^x_n\},\{\gamma^y_n\}$ be $\pi_1S$--translates of $\gamma$ that nest down on $x$ and $y$, respectively.   Then by \eqref{E:what is partial Phi}, we have
\[ \partial \Phi(x) = \bigcap_{n=1}^\infty \partial \mathscr{H}^+(\gamma^x_n) \quad \mbox{ and } \quad \partial \Phi(y) = \bigcap_{n=1}^\infty \partial \mathscr{H}^+(\gamma^y_n) \]
By Lemma~\ref{L:points identified by hat phi}, $\hat \Phi(\{\calL\} \times E)$ is a single point, which we denote $\hat \Phi(\{\calL\} \times E) = \calL_0  \in \EL^s(\dot S)$. Now observe that because $\gamma^x_n$ intersects $E$ for all sufficiently large $n$, \eqref{E:boundary intersections 2} implies
\[ \calL_0 \in \bigcap_{n=1}^\infty \hat \Phi(\{\calL\} \times \gamma^x_n) \subset \bigcap_{n=1}^\infty \partial \calX(\gamma^x_n) \subset \bigcap_{n=1}^\infty \partial \mathscr{H}^+(\gamma^x_n) = \partial \Phi(x).\]
Therefore, $\partial \Phi(x) = \calL_0$.  The exact same argument shows $\partial \Phi(y) = \calL_0$, and hence
\[ \partial \Phi(x) = \calL_0= \partial \Phi(y),\]
as required.

Now suppose $\partial \Phi(x) = \partial \Phi(y) = \calL_0 \in \EL^s(\dot S)$.
Again by \eqref{E:what is partial Phi} there are sequences $\{\gamma^x_n\}$ and $\{\gamma^y_n\}$ ($\pi_1S$-translates of $\gamma$) nesting down to $x$ and $y$  respectively so that
\[\bigcap_{n=1}^\infty \partial \mathscr{H}^{+}(\gamma^x_n)= \calL_0 = \bigcap_{n=1}^\infty \partial \mathscr{H}^{+}(\gamma^y_n). \]
Because the intersections are nested, this implies that for all $n$ we have
\[ \calL_0 \in \partial \mathscr{H}^{+}(\gamma_n^x) \cap \partial \mathscr{H}^{+}(\gamma_n^y). \]
Passing to a subsequence if necessary,  we may assume that for all $n$, $\mathscr{H}^+(\gamma_n^x) \subset \mathscr{H}^-(\gamma_n^y)$ and $\mathscr{H}^+(\gamma_n^y) \subset \mathscr{H}^-(\gamma_n^x)$.  Therefore, for all $n$ we have
\begin{eqnarray*} \calL_0 & \in & \partial \mathscr{H}^{+}(\gamma_n^x) \cap \partial \mathscr{H}^{+}(\gamma_n^y)\\
& = & \left( \partial \mathscr{H}^+(\gamma_n^x) \cap \partial \mathscr{H}^-(\gamma_n^y)\right)  \cap \left( \partial \mathscr{H}^+(\gamma_n^y) \cap \partial \mathscr{H}^-(\gamma_n^x) \right)\\
& = & \left( \partial \mathscr{H}^+(\gamma_n^x) \cap \partial \mathscr{H}^-(\gamma_n^x)\right)  \cap \left( \partial \mathscr{H}^+(\gamma_n^y) \cap \partial \mathscr{H}^-(\gamma_n^y) \right)\\
& = &  \partial \calX(\gamma_n^x) \cap \partial \calX(\gamma_n^y).
\end{eqnarray*}
The last equality here is an application of  \eqref{E:boundary intersections 1}.
Combining this with the description of $\calL_0$ above and \eqref{E:boundary intersections 2}, we have
\[ \partial \Phi(x) = \partial \Phi(y) = \calL_0 = \bigcap_{n=1}^\infty (\partial \calX(\gamma_n^x) \cap \partial \calX(\gamma_n^y)) = \bigcap_{n=1}^\infty (\hat \Phi(\partial \C(S) \times \gamma_n^x) \cap \hat \Phi(\partial \C(S) \times \gamma_n^y)). \]
For the last equation where we have applied \eqref{E:boundary intersections 2}, we have used the fact that the stable/unstable laminations of the pseudo-Anosov mapping classes corresponding to $\delta_n^x$ and $\delta_n^y$ in $\pi_1S$ stabilizing $\gamma_n^x$ and $\gamma_n^y$, respectively, are all distinct, hence $\calL_0$ is not one of the stable/unstable laminations.

From the equation above, we have $\calL_n^x,\calL_n^y \in \EL(S)$ and $x_n \in \gamma_n^x, y_n \in \gamma_n^y$ so that $\hat \Phi(\calL_n^x,x_n) = \hat \Phi(\calL_n^y,y_n) = \calL_0$, for all $n$.  According to Lemma~\ref{L:points identified by hat phi}, there exists $\calL \in \EL(S)$ so that $\calL_n^x = \calL_n^y = \calL$ for all $n$, and there exists a leaf or complementary region $E$ of $p^{-1}(\calL)$ so that $x_n,y_n \in E$.  Since $\gamma_n^x$ and $\gamma_n^y$ nest down on $x$ and $y$, respectively, it follows that $x_n \to x$ and $y_n \to y$ as $n \to \infty$.  Therefore, $x,y$ are endpoints of a leaf of $p^{-1}(\calL)$ or ideal endpoints of a complementary region of $p^{-1}(\calL)$, as required.
\end{proof}

We can now easily deduce the following, which also proves Proposition~\ref{P:CT difference}.

\begin{proposition} \label{P:preimage of witness ELs} Given $\calL_0 \in \EL^s(\dot S)$, $\partial \Phi^{-1}(\calL_0)$ is infinite if and only if $\calL_0 \in \EL(W)$ for some proper witness $W$.
\end{proposition}
\begin{proof}   Theorem~\ref{CTU} implies that for $\calL_0 \in \EL^s(\dot S)$, $\partial \Phi^{-1}(\calL_0)$ contains more than two points if and only if there is a lamination $\calL \in \EL(S)$ and a complementary region $U$ of $p^{-1}(\calL)$ so that $\partial \Phi^{-1}(\calL_0)$ is precisely the set of ideal points of $U$.  Moreover, in this case the proof above shows that $\calL_0 = \hat \Phi(\{\calL\} \times U)$.

On the other hand, Lemma~\ref{L:where the witness ELs come from} and the paragraph preceding it tell us that $\calL_0 \in \EL^s(\dot S)$ is contained in $\EL(W)$ for a proper witness $W \subsetneq \dot S$ if and only if it is given by $\calL_0 = \hat \Phi(\{\calL\} \times U)$ where $\calL \in \EL(S)$ and $U$ is the complementary region of $p^{-1}(\calL)$ containing $H_x$, where $x \in \calP$ with $W = \calW(x)$.

Finally, we note that a complementary region of a lamination $\calL \in \EL(S)$ has infinitely many ideal vertices if and only if it projects to a complementary region of $\calL$ containing a puncture, and this happens if and only if it contains a horoball $H_x$ for some $x \in \calP$.

Combining all three of the facts above proves the proposition.
\end{proof}

Now we define $\AbdC \subset \Abd$ to be those points that map by $\partial \Phi$ to $\EL(\dot S)$ and then define $\partial \Phi_0 \colon \AbdC \to \partial \C(\dot S) = \EL(\dot S)$ to be the ``restriction'' of $\partial \Phi$ to $\AbdC$.  Theorem~\ref{T:Phi0 identified} is a consequence of the Theorem~\ref{CTU} since $\partial \Phi_0$ is the restriction of $\partial \Phi$ to $\AbdC$.  Then Proposition~\ref{P:CT difference} is immediate from Proposition~\ref{P:preimage of witness ELs} and  the definitions. Theorem~\ref{T:UCT C short} then follows from Theorem~\ref{CT}.

We end with an alternate description of $\AbdC$.  For $\calL \in \EL(S)$, consider the subset $\calS_{\calL} \subset \partial \HH$ consisting of all ideal endpoints of complementary components of $p^{-1}(\calL)$ {\em which have infinitely many such ideal endpoints}.  That is, $\calS_{\calL}$ is the set of ideal endpoints of complementary regions that project to complementary regions of $\calL$ that contain a puncture.  The following is thus an immediate consequence of Theorem~\ref{CTU} and Proposition~\ref{P:preimage of witness ELs}.

\begin{corollary} \label{C:CT difference 2} The set  of points $\AbdC  \subset \Abd \subset \partial \HH$  that map to $\EL(\dot S) \subset \EL^s(\dot S)$ is
\[ \AbdC  = \Abd \setminus \bigcup_{\calL \in \EL(S)} \calS_{\calL}.\]
\hfill $\Box$
\end{corollary}

\bibliographystyle{alpha}
\bibliography{main}
\end{document}